\documentclass[a4paper]{article}

\usepackage{amsfonts}
\usepackage{amssymb}
\usepackage{amsthm}
\usepackage{amsmath}
\usepackage[top=4cm, bottom=3.5cm, left=2cm, right=2cm]{geometry}

\title{Recurrence and non-uniformity of bracket polynomials}
\date{}
\author{Matthew C. H. Tointon\thanks{For some of the time during which this work was carried out the author was supported by an EPSRC doctoral training grant, awarded by the Department of Pure Mathematics and Mathematical Statistics in Cambridge, and a Bye-Fellowship from Magdalene College, Cambridge.}}

\numberwithin{equation}{section}

\newtheorem{prop}{Proposition}[section]
\newtheorem{theorem}[prop]{Theorem}
\newtheorem{lemma}[prop]{Lemma}
\newtheorem{corollary}[prop]{Corollary}

\theoremstyle{definition}
\newtheorem{definition}[prop]{Definition}

\theoremstyle{remark}
\newtheorem{remark}[prop]{Remark}
\newtheorem{remarks}[prop]{Remarks}
\newtheorem{example}[prop]{Example}
\newtheorem*{remark*}{Remark}
\newtheorem*{remarks*}{Remarks}

\theoremstyle{plain}

\newcommand{\R}{\mathbb{R}}
\newcommand{\C}{\mathbb{C}}

\newcommand{\N}{\mathbb{N}}
\newcommand{\Z}{\mathbb{Z}}

\newcommand{\Prob}{\mathbb{P}}
\newcommand{\E}{\mathbb{E}}

\newcommand{\g}{\mathfrak{g}}
\newcommand{\h}{\mathfrak{h}}

\newcommand{\Lip}{\text{\textup{Lip}}}
\newcommand{\Span}{\text{\textup{Span}}}
\newcommand{\id}{\text{Id}}

\begin{document}
\maketitle
\begin{center}
\scriptsize{\textit{Department of Pure Mathematics and Mathematical Statistics, Centre for Mathematical Sciences, University of Cambridge, Wilberforce Road, Cambridge CB3 0WB, United Kingdom}}\\
\scriptsize{\textit{email: M.Tointon@dpmms.cam.ac.uk}}
\end{center}
\begin{abstract}
In his celebrated proof of Szemer\'edi's theorem that a set of integers of positive density contains arbitrarily long arithmetic progressions, W. T. Gowers introduced a certain sequence of norms $\|\,\cdot\,\|_{U^2[N]}\le\|\,\cdot\,\|_{U^3[N]}\le\ldots$ on the space of complex-valued functions on the set $[N]$. An important question regarding these norms concerns for which functions they are `large' in a certain sense.

This question has been answered fairly completely by B. Green, T. Tao and T. Ziegler in terms of certain algebraic functions called \emph{nilsequences}. In this work we show that more explicit functions called \emph{bracket polynomials} have `large' Gowers norm. Specifically, for a fairly large class of bracket polynomials, called \emph{constant-free} bracket polynomials, we show that if $\phi$ is a bracket polynomial of degree $k-1$ on $[N]$ then the function $f:n\mapsto e(\phi(n))$ has Gowers $U^k[N]$-norm uniformly bounded away from zero.

We establish this result by first reducing it to a certain recurrence property of sets of constant-free bracket polynomials. Specifically, we show that if $\theta_1,\ldots,\theta_r$ are constant-free bracket polynomials then their values, modulo $1$, are all close to zero on at least some constant proportion of the points $1,\ldots,N$.

The proof of this statement relies on two deep results from the literature. The first is work of V. Bergelson and A. Leibman showing that an arbitrary bracket polynomial can be expressed in terms of a so-called \emph{polynomial sequence} on a nilmanifold. The second is a theorem of B. Green and T. Tao describing the quantitative distribution properties of such polynomial sequences.

In the special cases of the bracket polynomials $\phi_{k-1}(n)=\alpha_{k-1}n\{\alpha_{k-2}n\{\ldots\{\alpha_1n\}\ldots\}\}$, with $k\le5$, we give elementary alternative proofs of the fact that $\|\phi_{k-1}\|_{U^k[N]}$ is `large', without reference to nilmanifolds. Here we write $\{x\}$ for the fractional part of $x$, chosen to lie in $(-1/2,1/2]$.
\end{abstract}
\tableofcontents
\section{Introduction}
A remarkable theorem of E. Szemer\'edi \cite{szem} states that, for $\sigma>0$, $k\in\N$ and $N\gg_{k,\sigma}1$, every subset of $[N]$ of cardinality at least $\sigma N$ contains a $k$-term arithmetic progression. The first good bounds in this theorem were obtained in the celebrated proof of Szemer\'edi's theorem by W. T. Gowers \cite{gowers.szem}.

A key observation in Gowers's work was that arithmetic progressions in a finite abelian group $G$ can be detected using certain norms $\|\,\cdot\,\|_{U^2(G)}\le\|\,\cdot\,\|_{U^3(G)}\le\ldots$ on the space of complex-valued functions on $G$. In general, the $U^k(G)$-norm is helpful in detecting arithmetic progressions of length $k+1$ in the group $G$, and this has led these norms to become one of the major tools in additive combinatorics. Gowers called them \emph{uniformity norms}; they are now often called \emph{Gowers uniformity norms}, or simply \emph{Gowers norms}.

Given a finite abelian group $G$, the Gowers norms $\|\,\cdot\,\|_{U^k(G)}$ are defined as follows. First, define the \emph{multiplicative derivative} of a function $f:G\to\C$ by
\[
\Delta_h^*f(x):=f(x+h)\overline{f(x)},
\]
and abbreviate
\[
\Delta_{h_1,\ldots,h_k}^*f(x):=\Delta_{h_1}^*\ldots\Delta_{h_k}^*f(x).
\]
Then for each integer $k\ge2$ define the \emph{Gowers $U^k(G)$-norm} by
\[
\|f\|_{U^k(G)}:=(\E_{x,h_1,\ldots,h_k\in G}\Delta_{h_1,\ldots,h_k}^*f(x))^{1/2^k}.
\]
It can be shown that $\|\cdot\|_{U^k(G)}$ is indeed a norm, but we will not need this fact and so we omit its proof.

Szemer\'edi's theorem, of course, concerns arithmetic progressions in $\Z$ or, more precisely, in $[N]:=\{1,\ldots,N\}$, neither of which is a finite group. However, it is also possible to define the Gowers norm of a function $f:[N]\to\C$, and this can then be applied in finding arithmetic progressions inside $[N]$.  The following definition is reproduced from \cite[\S1]{u4.inverse}.
\begin{definition}[Gowers $U^k{[N]}$-norm]\label{def:[N]-norm}
Given an integer $N>0$ fix some other integer $\tilde{N}\ge2^kN$. Define a function $\tilde{f}:\Z/\tilde{N}\Z\to\C$ by $\tilde{f}(x)=f(x)$ for $x\in[N]$ and $\tilde{f}(x)=0$ otherwise. Then $\|f\|_{U^k[N]}$ is defined by
\begin{equation}\label{eq:ukn.norm}
\|f\|_{U^k[N]}:=\|\tilde{f}\|_{U^k(\Z/\tilde{N}\Z)}/\|1_{[N]}\|_{U^k(\Z/\tilde{N}\Z)}.
\end{equation}
Here, and throughout the present work, if $X$ is a set then $1_X$ denotes the indicator function of $X$.
\end{definition}
As is remarked in \cite[\S1]{u4.inverse}, it is easy to see that the quantity (\ref{eq:ukn.norm}) is independent of the choice of $\tilde{N}\ge2^kN$, and so $\|\cdot\|_{U^k[N]}$ is well defined.

Denote by $\mathcal{D}$ the unit disc $\{z\in\C:|z|\le1\}$. It turns out that when applying Gowers norms to finding arithmetic progressions in $[N]$ it is useful to have a classification of functions $f:[N]\to\mathcal{D}$ satisfying
\begin{equation}\label{eq:non-uniform}
\|f\|_{U^k[N]}\ge\delta.
\end{equation}
A function satisfying (\ref{eq:non-uniform}) for a given $\delta$ is generally said to be \emph{non-uniform}; a classification of such functions is the content of so-called \emph{inverse conjectures} and \emph{inverse theorems} for the Gowers norms.

It is easy to see that $\|f\|_{U^k[N]}$ is bounded above by $1$ for every function $f:[N]\to\mathcal{D}$, and also that this bound is attained by the function $1_{[N]}$. In fact, there is a very natural broader class of functions attaining this upper bound, which we now describe. Given a function $\phi:[N]\to\R$ we denote the \emph{discrete derivatives} of $\phi$ by
\[
\Delta_h\phi(n):=\phi(n+h)-\phi(n),
\]
and abbreviate
\[
\Delta_{h_1,\ldots,h_n}\phi:=\Delta_{h_1}\ldots\Delta_{h_n}\phi.
\]
Adopting the standard convention that $e(x):=\exp(2\pi ix)$, we have
\begin{equation}\label{eq:mult.deriv}
\Delta_h^*e(\phi(x))=e(\Delta_h\phi(x)).
\end{equation}
When $\phi$ is a polynomial of degree $k-1$, this implies in particular that every term in the sum
\begin{equation}\label{eq:avg.deriv}
\sum_{n,h_1,\ldots,h_{k+1}}\Delta_{h_1,\ldots,h_{k+1}}^*e(\phi(n))
\end{equation}
is equal to 1, and so if $f:[N]\to\mathcal D$ is the function defined by setting $f(n):=e(\phi(n))$ then the Gowers norm $\|f\|_{U^k[N]}$ is equal to 1.

There are more exotic examples of functions $f:[N]\to\mathcal{D}$ satisfying (\ref{eq:non-uniform}). B. Green, T. Tao and T. Ziegler \cite[Proposition 1.4]{u4.inverse} show that certain algebraic functions called \emph{nilsequences}, which we define shortly, are non-uniform. Indeed, they demonstrate the stronger fact that any function that correlates with a nilsequence in a certain sense is non-uniform.

In order to define a nilsequence we must first recall that an \emph{$s$-step nilmanifold} is the quotient $G/\Gamma$ of an $s$-step nilpotent Lie group $G$ by a discrete cocompact subgroup $\Gamma$. For example, if $G$ is the \emph{Heisenberg group}
\[
\left(\begin{array}{ccc}
               1 & \R & \R \\
               0 & 1  & \R \\
               0 & 0   & 1
               \end{array}\right)
\]
and $\Gamma$ is the discrete subgroup
\[
\left(\begin{array}{ccc}
              1 & \Z & \Z \\
               0 & 1  & \Z \\
               0 & 0   & 1
               \end{array}\right)
\]
then it is straightforward to check that $G/\Gamma$ has a fundamental domain in $G$ defined by
\begin{equation}\label{eq:heis.fund}
\left(\begin{array}{ccc}
               1 & (-1/2,1/2] & (-1/2,1/2]  \\
               0 & 1  & (-1/2,1/2]  \\
               0 & 0  & 1
               \end{array}\right);
\end{equation}
see, for example, \cite[\S1]{poly.seq}. The subgroup $\Gamma$ is therefore cocompact, and so $G/\Gamma$ is a 2-step nilmanifold, called the \emph{Heisenberg nilmanifold}.

\emph{Throughout this paper, when we write that $G/\Gamma$ is a nilmanifold we assume that $G$ is a connected, simply connected nilpotent Lie group.}

A sequence $y_n$ is said to be an \emph{$s$-step nilsequence} if there exists an $s$-step nilmanifold $G/\Gamma$, elements $g,x\in G$ and a continuous function $F:G/\Gamma\to\C$ such that
\[
y_n=F(g^nx\Gamma).
\]

It turns out that this exhausts all the possibilities for non-uniform functions. Indeed, a remarkable inverse theorem for the Gowers norms, also due to Green, Tao and Ziegler \cite{uk.inverse}, states, roughly, that if $f:[N]\to\mathcal{D}$ satisfies (\ref{eq:non-uniform}) then $f$ correlates with a $(k-1)$-step nilsequence. We refer the reader to \cite{uk.inverse} for a precise formulation.

Thus we have a comprehensive, if not particularly explicit, classification of all the Gowers non-uniform functions $f:[N]\to\mathcal {D}$. It is noted in \cite[\S1]{uk.inverse}, however, that more explicit formulations of the inverse conjectures for the Gowers norms are also possible. In the case of a cyclic group $\Z/N\Z$ of prime order, for example, \cite[Theorem 10.9]{u3.inverse} gives a particularly concrete inverse theorem for the Gowers $U^3$-norm.

In order to describe that result we require some notation. Here, and throughout this work, we denote by $\{x\}$ the fractional part of $x\in\R$, chosen to lie in $(-1/2,1/2]$, and denote by $[x]$ the integer part $[x]:=x-\{x\}$. Let $N$ be a prime and let $k\ge0$. Then \cite[Theorem 10.9]{u3.inverse} says, roughly, that a function $f:\Z/N\Z\to\mathcal{D}$ satisfies $\|f\|_{U^3(\Z/N\Z)}\ge\delta$ if and only if there are $\xi_1,\xi_2\in\widehat{\Z/N\Z}$ and a real number $\alpha$ such that $f$ correlates with the function $f':\Z/N\Z\to\mathcal{D}$ defined by
\[
f'(x):=e(\alpha\{\xi_1\cdot x\}\{\xi_2\cdot x\}).
\]
Again, we refer the reader to \cite{u3.inverse} for a precise statement.

The function $x\mapsto\alpha\{\xi_1\cdot x\}\{\xi_2\cdot x\}$ is an example of a so-called \emph{bracket polynomial} on $\Z/N\Z$. It is also possible to define bracket polynomials on the set $[N]$, where $N$ is now an arbitrary positive integer. Essentially these are functions like $\phi:n\mapsto 2n\{\sqrt{2}n^2\}+n^3$ that are constructed from genuine polynomials using the operations $+,\cdot,\{\cdot\}$; we give a precise definition in Section~\ref{sec:def.of.bracket}, and in particular clarify the notion of the \emph{degree} of a bracket polynomial.

It turns out that bracket polynomials on $[N]$ arise quite naturally from sequences on nilmanifolds. To see this in the case of the bracket polynomial $\{\alpha n[\beta n]\}$, for example, let $g(n)$ be the sequence in the Heisenberg group given by
\[
g(n)=\left(\begin{array}{ccc}
               1 & -\alpha n & 0 \\
               0 & 1        & \beta n \\
               0 & 0        & 1
               \end{array}\right).
\]
It is straightforward to check that the image of $g(n)$ in the fundamental domain (\ref{eq:heis.fund}) is the element
\[
\left(\begin{array}{ccc}
               1 & \{-\alpha n\} & \{\alpha n[\beta n]\} \\
               0 & 1            & \{\beta n\}\\
               0 & 0            & 1
               \end{array}\right),
\]
in which $\{\alpha n[\beta n]\}$ appears quite prominently as the upper-right entry. This in fact turns out to be a general phenomenon. Bergelson and Leibman \cite{berg.leib} show that an arbitrary bracket polynomial can be expressed in terms of a nilmanifold in similar fashion. See Theorem \ref{thm:berg.leib} for more details.

Given the role played by bracket polynomials on $\Z/N\Z$ in the inverse theory for the $U^3(\Z/N\Z)$-norms, as well as the link between bracket polynomials and sequences on nilmanifolds due to Bergelson--Leibman and the link between sequences on nilmanifolds and Gowers norms due to Green--Tao--Ziegler, it is natural to ask what role bracket polynomials on $[N]$ play in the inverse theory for the $U^k[N]$ norms. The aim of this paper is to explore this role.

Our first, and principal, theorem states that a fairly large class of bracket polynomials $\phi$ give rise to non-uniform functions $e(\phi)$. This is the class of \emph{constant-free} bracket polynomials. These are defined precisely in Section~\ref{sec:def.of.bracket}, but essentially a bracket polynomial is said to be constant free if it is constructed from genuine polynomials using the operations $+,\cdot,\{\,\cdot\,\}$, and each of the genuine polynomials used in this construction has zero constant term. Thus, for example, the bracket polynomial $\{\alpha n\{\beta n\}\}+\gamma n^2$ is constant free, but the bracket polynomial $\alpha n\{\beta n+\gamma\}$ is not.

We show, then, that if $\phi$ is a constant-free bracket polynomial of degree at most $k-1$ then the quantity $\|e(\phi)\|_{U^k[N]}$ must be bounded away from zero. The bound obtained is uniform in $N$. It does depend on the bracket polynomial being considered, but only on its `shape'; we make this precise in Section~\ref{sec:def.of.bracket} using the notion of a \emph{bracket form}, but essentially this means, for example, that the bound obtained for the bracket polynomial $\{\alpha n\{\beta n\}\}$ is uniform across all choices of $\alpha$ and $\beta$.
\begin{theorem}[Bracket polynomials are non-uniform; rough statement]\label{thm:bracket.non.uniform.vague}
Let $\phi$ be a constant-free bracket polynomial of degree at most $k-1$. Then there is some $\delta$, depending only on the `shape' of $\phi$, such that
\[
\|e(\phi)\|_{U^k[N]}\ge\delta.
\]
\end{theorem}
See Theorem \ref{thm:bracket.non.uniform} for a more precise statement. The `constant-free' condition results from our use of Theorem \ref{lem:general.Bohr.set.size.bound.lower}; see also Remark \ref{rem:const.free}.

At its simplest level, our proof of Theorem \ref{thm:bracket.non.uniform.vague} rests on a preliminary result stating that bracket polynomials have a rather suggestive property called being \emph{locally polynomial}.
\begin{definition}[Locally polynomial]
Let $\phi:[N]\to\R$ be a function and let $B\subset[N]$. Then $\phi$ is said to be \emph{locally polynomial of degree $k-1$} on $B$ if whenever $n\in[N]$ and $h\in[-N,N]^k$ satisfy $n+\omega\cdot h\in B$ for all $\omega\in\{0,1\}^k$ we have $\Delta_{h_1,\ldots,h_k}\phi(n)=0$.
\end{definition}
Indeed, we show in Section \ref{sec:brackets.locally.poly} that every bracket polynomial $\phi$ on $[N]$ is locally polynomial on some `large' set $B_\phi\subset[N]$. The relevance of this to the study of Gowers norms lies in the identity (\ref{eq:mult.deriv}). Just as this identity implied that the terms in the sum (\ref{eq:avg.deriv}) were all equal to $1$ when $\phi$ was a genuine polynomial, if $\phi$ is \emph{locally} polynomial on a suitably large set then this suggests some bias towards $1$ in the terms of the sum (\ref{eq:avg.deriv}). This in turn suggests that $\|e(\phi)\|_{U^k[N]}$ should be bounded away from zero.

Unfortunately, it is not clear that one can proceed directly from the property of being locally polynomial to the property of having large Gowers norm, and so the results of Section \ref{sec:brackets.locally.poly} alone are not sufficient to prove Theorem \ref{thm:bracket.non.uniform.vague}. In Section \ref{sec:strong.local.poly}, however, we show that a slightly stronger property, which we call being \emph{strongly locally polynomial}, is sufficient to imply that a bracket polynomial is non-uniform.

It turns out that a certain recurrence property of bracket polynomials is sufficient to imply the property of being strongly locally polynomial, and hence to imply that a bracket polynomial is non-uniform. This is the principal motivation for our second theorem.
\begin{theorem}[Recurrence of bracket polynomials; rough statement]\label{thm:general.Bohr.set.size.bound.lower.rough}
Let $\theta_1,\ldots,\theta_r$ be con\-stant-free bracket polynomials and let $\delta>0$. Then there are some $\varepsilon>0$ depending only on the `shapes' of the $\theta_i$, and $N_0>0$ depending on the `shapes' of the $\theta_i$ and on $\delta$, such that whenever $N\ge N_0$ the proportion of $n\in[N]$ for which $\{\theta_i(n)\}\in(-\delta,\delta)$ is at least $\varepsilon$.
\end{theorem}
See Theorem \ref{lem:general.Bohr.set.size.bound.lower} for a precise statement. We show in Section \ref{sec:strong.local.poly} that this is sufficient to imply Theorem \ref{thm:bracket.non.uniform.vague}; it is potentially also of interest in its own right.
\begin{remark}
We show by example in Remark \ref{rem:const.free} that the constant-free condition is necessary in Theorem \ref{thm:general.Bohr.set.size.bound.lower.rough}. Theorem \ref{thm:bracket.non.uniform.vague}, on the other hand, could conceivably remain true in the absence of that condition.
\end{remark}

In Section \ref{sec:literature}, we appeal to the work of Bergelson--Leibman showing that bracket polynomials can be expressed in terms of certain sequences on nilmanifolds, as well as to work of B. Green and T. Tao describing the distribution properties of such sequences, to establish Theorem \ref{thm:general.Bohr.set.size.bound.lower.rough}, or rather the more precise Theorem \ref{lem:general.Bohr.set.size.bound.lower}.

The appeal to the results of Bergelson--Leibman and Green--Tao in the proof of Theorem \ref{thm:bracket.non.uniform.vague} renders the argument far from elementary. It is interesting to see for which bracket polynomials one can use elementary methods to establish Theorem \ref{thm:bracket.non.uniform.vague}. In Sections \ref{sec:weak.recur} and \ref{sec:approx.loc.poly} we consider this problem in the model setting of the bracket polynomials $\phi_{k-1}(n)$ defined by $\alpha_{k-1}n\{\alpha_{k-2}n\{\ldots\{\alpha_1n\}\ldots\}\}$. Theorem \ref{thm:bracket.non.uniform.vague} of course instantly tells us that $\|e(\phi_{k-1})\|_{U^k[N]}\gg_k1$; in Sections \ref{sec:weak.recur} and \ref{sec:approx.loc.poly} we arrive at this statement in the cases $k\le5$ by entirely elementary methods.
\subsection*{Acknowledgements}It is a pleasure to thank Tim Gowers and Ben Green for helpful and stimulating conversations, and Emmanuel Breuillard, Tom Sanders and an anonymous referee for careful readings of and detailed comments on earlier versions of this paper.
\section{Bracket polynomials on $[N]$}\label{sec:def.of.bracket}
In this section we give formal definitions of some of the concepts we discussed in the introduction. In particular, we define bracket polynomials on $[N]$ precisely. The definitions of bracket polynomials are essentially already contained in the literature; see, for example, \cite[1.11-1.12]{berg.leib}. Nonetheless, we give them in full detail here, in part so as to set notation, but also in order to introduce the related concept of a \emph{bracket form}. The latter is necessary in order to make precise what we mean by `shape' in Theorem \ref{thm:bracket.non.uniform.vague}.
\begin{definition}[Bracket polynomials on ${[N]}$]\label{def:bracket.poly}
Bracket polynomials on $[N]$ are functions from $[N]$ to $\R$ defined recursively as follows.
\begin{itemize}
\item A genuine polynomial $\phi$ of degree $k$ is also a bracket polynomial of degree at most $k$.
\item If $\phi:[N]\to\R$ is a bracket polynomial of degree at most $k$ then the functions $-\phi:[N]\to\R$, defined by $(-\phi)(n):=-(\phi(n))$, and $\{\phi\}:[N]\to\R$, defined by $\{\phi\}(n):=\{\phi(n)\}$, are also bracket polynomials of degree at most $k$.
\item If $\phi_1,\phi_2:[N]\to\R$ are bracket polynomials of degree at most $k_1,k_2$, respectively, then the function $\phi_1\cdot\phi_2:[N]\to\R$ defined by $\phi_1\cdot\phi_2(n):=\phi_1(n)\phi_2(n)$ is a bracket polynomial of degree at most $k_1+k_2$, and the function $\phi_1+\phi_2:[N]\to\R$ defined by $\phi_1+\phi_2(n):=\phi_1(n)+\phi_2(n)$ is a bracket polynomial of degree at most $\max\{k_1,k_2\}$.
\end{itemize}
If in this definition we restrict the genuine polynomials to those with zero constant term, the resulting functions are said to be \emph{constant-free} bracket polynomials. Those bracket polynomials that do not use the $+$ operation are called \emph{elementary}.
\end{definition}
\begin{remark}It will almost always be the case that we will be interested only in the value of a bracket polynomial modulo 1, and so we might easily and naturally define bracket polynomials to be functions into $\R/\Z$. However, certain statements and proofs are slightly cleaner if we view them as functions into $\R$ and then project to $\R/\Z$ only when it comes to the final application.
\end{remark}
In order to make the statement of Theorem \ref{thm:bracket.non.uniform.vague} precise, we need some way of defining what we mean by the `shape' of a bracket polynomial. To that end, we first develop a definition that formalises this concept for genuine polynomials. The basic idea is to say that two polynomials $\phi_1$ and $\phi_2$ have the same `shape' if there is some `polynomial' in $n$ with coefficients taken from the list of symbols $\alpha_1,\alpha_2,\ldots$ such that both $\phi_1$ and $\phi_2$ can be obtained by replacing each symbol $\alpha_i$ by a real number. Thus, for example, the polynomials $3n^2$ and $\pi n^2$ would have the same `shape' because they can each be realised by replacing the symbol $\alpha_1$ in the `polynomial' $\alpha_1n^2$ by a real number. We shall call $\alpha_1n^2$ a \emph{polynomial form} and call $3n^2$ and $\pi n^2$ \emph{realisations} of the polynomial form $\alpha_1n^2$.

In fact, the definition of a polynomial form will need to be slightly more complicated than is suggested by the preceding paragraph. This is because in Appendix \ref{sec:berg.leib} it will be convenient for the set of polynomial forms to form a ring.
\begin{definition}[Ring of polynomial forms]\label{def:poly.form}
Let $\alpha_1,\alpha_2,\ldots$ be a countably infinite list of symbols; we shall call this list an \emph{alphabet}. Define a \emph{monomial form} in these symbols to be a string of the form $+\alpha_{i_1}\cdots\alpha_{i_t}n^k$ or $-\alpha_{i_1}\cdots\alpha_{i_t}n^k$, with $i_1\le\ldots\le i_t$ and $k\ge0$ an integer. Define $k$ to be the \emph{degree} of such a monomial form. If $k\ne0$ then we additionally say that $+\alpha_{i_1}\cdots\alpha_{i_t}n^k$ and $-\alpha_{i_1}\cdots\alpha_{i_t}n^k$ are \emph{constant-free} monomial forms of degree $k$.

Now suppose that $\Phi_1,\ldots,\Phi_r$ is a finite list of monomial forms of degree at most $k$. Then the string $\Phi_1+\ldots+\Phi_r$ is said to be a \emph{polynomial form} of degree at most $k$. If the $\Phi_i$ are all constant free then the string $\Phi_1+\ldots+\Phi_r$ is also said to be \emph{constant free}.

We make the set of polynomial forms into a ring by defining multiplication on the set of monomial forms, and then extending it (uniquely) to the set of polynomial forms by requiring it to be distributive over addition. Specifically, we formally define multiplication on the symbols $+$ and $-$ by setting $(+\cdot+)=(-\cdot-)=+$ and $(+\cdot-)=(-\cdot+)=-$, and then if each of $\epsilon,\epsilon'$ represents either $+$ or $-$ we define the product of the monomial forms $\epsilon\alpha_{i_1}\cdots\alpha_{i_t}n^k$ and $\epsilon'\alpha_{{i'}_1}\cdots\alpha_{{i'}_{t'}}n^{k'}$ to be
\[
\epsilon\alpha_{i_1}\cdots\alpha_{i_t}n^k\cdot\epsilon'\alpha_{{i'}_1}\cdots\alpha_{{i'}_{t'}}n^{k'}=(\epsilon\cdot\epsilon')\alpha_{j_1}\cdots\alpha_{j_{t+t'}}n^{k+k'},
\]
where the $\alpha_{j_1},\ldots,\alpha_{j_{t+t'}}$ are precisely the $\alpha_{i_1},\ldots,\alpha_{i_t},\alpha_{i'_1},\ldots,\alpha_{i'_{t'}}$, only permuted so that $j_l\le j_{l'}$ whenever $l<l'$.

From now on we drop the $+$ from the polynomial form $+\alpha_{i_1}\cdots\alpha_{i_t}n^k$ and write simply $\alpha_{i_1}\cdots\alpha_{i_t}n^k$.
\end{definition}
Thus, for example, the strings $-\alpha_1n$ and $-\alpha_2n^3$ are polynomial forms, and their product is $\alpha_1\alpha_2n^4$.
\begin{definition}[Realisation of a polynomial form]
Suppose that $\Phi$ is a polynomial form featuring the symbols $\alpha_1,\ldots,\alpha_m$, and for each $i=1,\ldots,m$ let $a_i$ be a real number. Then the polynomial obtained by replacing each instance of $\alpha_i$ in $\Phi$ with the real number $a_i$ is said to be a \emph{realisation} of the polynomial form $\Phi$.
\end{definition}
Thus, for example, the string $\Phi=\alpha_1n+\alpha_2n^2-\alpha_1\alpha_2n^3$ is a polynomial form of degree at most 3, and the polynomial $2n+3n^2-6n^3$ is a realisation of $\Phi$.

To convert this into a definition valid for bracket polynomials, let us make the following slightly more abstract version of Definition \ref{def:bracket.poly}.
\begin{definition}[Bracket expressions]\label{def:br.exp}
Let $A$ be a set. Then the \emph{bracket expressions} in the elements of $A$ are certain strings of elements of $a$ and the symbols $+,-,\cdot,\{,\}$, defined recursively as follows.
\begin{itemize}
\item If $a\in A$ then the string $a$ is a bracket expression in the elements of $A$.
\item If $b$ is a bracket expression in the elements of $A$ then the strings $-b$ and $\{b\}$ are also a bracket expression in the elements of $A$.
\item If $b_1,b_2$ are bracket expression in the elements of $A$ then the strings $b_1\cdot b_2$ and $b_1+b_2$ are bracket expressions in the elements of $A$.
\end{itemize}
\end{definition}
\begin{remark}
Definition \ref{def:br.exp} is similar in spirit to that found in \cite[\S6.3]{berg.leib}.
\end{remark}
\begin{definition}[Bracket forms]
Let $A$ be the set of polynomial forms in the alphabet $\alpha_1,\alpha_2,\ldots$. Then the \emph{bracket forms} in the same alphabet are certain bracket expressions in the elements of $A$ defined recursively as follows.
\begin{itemize}
\item A polynomial form of degree at most $k$ is also a bracket form of degree at most $k$.
\item If $\Phi$ is a bracket form of degree at most $k$ then the expressions $-\Phi$ and $\{\Phi\}$ are also bracket forms of degree at most $k$.
\item If $\Phi_1,\Phi_2$ are bracket forms of degree at most $k_1,k_2$, respectively, then the expression $\Phi_1\cdot\Phi_2$ is a bracket form of degree at most $k_1+k_2$, and the expression $\Phi_1+\Phi_2$ is a bracket polynomial of degree at most $\max\{k_1,k_2\}$.
\end{itemize}
If the polynomial forms appearing in this recursive construction of a bracket form are all constant free, then the resulting bracket form is also said to be \emph{constant free}.

Now suppose that $\Phi$ is a bracket form featuring the symbols $\alpha_1,\ldots,\alpha_m$, and for each $i=1,\ldots,m$ let $a_i$ be a real number. Then the bracket polynomial obtained by replacing each instance of $\alpha_i$ in $\Phi$ with the real number $a_i$ is said to be a \emph{realisation} of the bracket form $\Phi$.
\end{definition}
Thus, for example, the string $\Phi=\{\alpha_1n\{\alpha_2n\}\}$ is a constant-free bracket form of degree at most 2, and the bracket polynomials $\{\sqrt{2}n\{\sqrt{3}n\}\}$ and $\{\pi n\{\frac{2}{7}n\}\}$ are realisations of $\Phi$.

We are now in a position to state a more precise version of Theorem \ref{thm:bracket.non.uniform.vague}.
\begin{theorem}[Bracket polynomials are non-uniform; precise statement]\label{thm:bracket.non.uniform}
Let $\Phi$ be a constant-free bracket form of degree at most $k-1$. Then for every realisation $\phi$ of $\Phi$ we have
\[
\|e(\phi)\|_{U^k[N]}\gg_\Phi1.
\]
\end{theorem}

When considering a bracket polynomial such as $\phi:n\mapsto n\{\alpha n\{\beta n\}\{\gamma n\}\}$ it will be useful to have a way of referring to the simpler `bracketed' components, in this case $\{\beta n\}$, $\{\gamma n\}$ and $\{\alpha n\{\beta n\}\{\gamma n\}\}$, from which $\phi$ is built up. We shall therefore call $\{\beta n\}$, $\{\gamma n\}$ and $\{\alpha n\{\beta n\}\{\gamma n\}\}$ the \emph{bracket components} of $\phi$. In general, we shall use the following definition, which is similar, though not identical, to that found in \cite[\S8]{berg.leib}.
\begin{definition}[Bracket components of a bracket polynomial]
The set $C(\phi)$ of \emph{bracket components} of a bracket polynomial $\phi$ will be a set of bracket polynomials defined recursively as follows.
\begin{itemize}
\item If $\phi$ is a genuine polynomial then $C(\phi)=\varnothing$.
\item If $\phi=-\nu$ for some bracket polynomial $\nu$ then $C(\phi)=C(\nu)$.
\item If $\phi=\{\nu\}$ for some bracket polynomial $\nu$ then $C(\phi)=C(\nu)\cup\{\nu\}$.\footnote{Note that $\{\nu\}$ here is the singleton containing $\nu$, not the fractional part of $\nu$.}
\item If $\phi=\nu_1\cdot\nu_2$ or $\nu_1+\nu_2$ for some bracket polynomials $\nu_1,\nu_2$ then $C(\phi)=C(\nu_1)\cup C(\nu_2)$.
\end{itemize}
The set $C(\Phi)$ of \emph{bracket components} of a bracket form $\Phi$ will be a set of bracket forms defined analogously, as follows.
\begin{itemize}
\item If $\Phi$ is a polynomial form then $C(\Phi)=\varnothing$.
\item If $\Phi=-\Theta$ for some bracket form $\Theta$ then $C(\Phi)=C(\Theta)$.
\item If $\Phi=\{\Theta\}$ for some bracket form $\Theta$ then $C(\Phi)=C(\Theta)\cup\{\Theta\}$.
\item If $\Phi=\Theta_1\cdot\Theta_2$ or $\Theta_1+\Theta_2$ for some bracket forms $\Theta_1,\Theta_2$ then $C(\Phi)=C(\Theta_1)\cup C(\Theta_2)$.
\end{itemize}
\end{definition}
\section{Bracket polynomials are locally polynomial}\label{sec:brackets.locally.poly}
We discussed in the introduction the relevance to the Gowers norms of the property of being locally polynomial. The aim of this section is to show that bracket polynomials have this property.
\begin{prop}[A bracket polynomial is locally polynomial on a set of positive density]
\label{prop:brackets.locally.poly}
Let $\phi:[N]\mapsto\R$ be a bracket polynomial of degree at most $k$ with $|C(\phi)|\le c$. Then there exists a set $B\subset[N]$ of cardinality $\Omega_{k,c}(N)$ on which $\phi$ is locally polynomial of degree at most $k$.
\end{prop}
The proof of this is straightforward, but will motivate much of what comes later. Before we embark on the main body of the proof, let us record the following trivial, but repeatedly useful, properties of fractional parts.
\begin{lemma}\label{lem:frac.part.difference}
Let $x,y\in\R/\Z$ and let $\delta<1/2$, and suppose that $\{x\}$ and $\{y\}$ both lie in a subinterval $I\subset(-1/2,1/2]$ of width $\delta$. Then
\begin{enumerate}
\renewcommand{\labelenumi}{(\roman{enumi})}
\item $\{x\}-\{y\}=\{x-y\}$, and this quantity lies in the interval $(-\delta,\delta)$;
\item if $I$ is centred on $0$ then we may additionally conclude that $\{x\}+\{y\}=\{x+y\}$ and that this quantity also lies in the interval $(-\delta,\delta)$.
\end{enumerate}
\end{lemma}
The proof of Proposition \ref{prop:brackets.locally.poly} is essentially in two parts. In the first (more substantial) part of the argument we identify a collection of sets on which $\phi$ is locally polynomial.
\begin{prop}\label{prop:local.poly.set}
Suppose that $\phi:[N]\to\R$ is a bracket polynomial of degree at most $k$. Then there exists a parameter $\delta\gg_k1$ such that if for each $\nu\in C(\phi)$ we have an interval $J_\nu$ of width at most $\delta$ inside $(-1/2,1/2]$ then $\phi$ is locally polynomial of degree $k$ on the set
\begin{displaymath}
\{n\in[N]:\{\nu(n)\}\in J_\nu\text{ for all $\nu\in C(\phi)$}\}.
\end{displaymath}
\end{prop}
The second part of the proof of Proposition~\ref{prop:brackets.locally.poly} is a simple pigeonholing argument, which we present as a lemma for ease of later reference, showing that at least one set of the form given by Propopsition~\ref{prop:local.poly.set} must have cardinality $\Omega_{k,c}(N)$.
\begin{lemma}\label{lem:pigeonhole.mod.1}
Let $I$ be an interval in $\R$, let $A\subset\N$ be a finite set and let $g_1,\ldots,g_l:A\to I$ be functions. Let $\delta_1,\ldots,\delta_l<|I|$. Then there exist subintervals $J_1,\ldots,J_l\subset I$ of widths $\delta_1,\ldots,\delta_l$, respectively, such that
\begin{displaymath}
|\{n\in A:g_i(n)\in J_i\text{ for $i=1,\ldots,l$}\}|\gg_{l}\frac{\delta_1\cdots\delta_l|A|}{|I|^l}.
\end{displaymath}
\end{lemma}
\begin{proof}This is essentially contained in the first part of the proof of \cite[Lemma 4.20]{tao-vu}. For each $i=1,\ldots,l$ we divide $I$ into $\lceil|I|/\delta_i\rceil$ subintervals: $\lfloor|I|/\delta_i\rfloor$ of length $\delta_i$, and at most one, the remainder, of length less than $\delta_i$.

Taking all products of these subintervals inside $I^l$, we divide $I^l$ into $\prod_{i=1}^l\lceil|I|/\delta_i\rceil$ boxes of side lengths at most $\delta_1,\ldots,\delta_l$. By the pigeonhole principle, one of these boxes must contain the images of at least
\begin{equation}\label{eq:A3608}
\frac{|A|}{\prod_{i=1}^l\lceil|I|/\delta_i\rceil}
\end{equation}
elements of $A$ under the map
\[
(g_1,\ldots,g_l):A\to I^l.
\]
Now (\ref{eq:A3608}) is certainly at least
\[
\frac{\delta_1\cdots\delta_l|A|}{2^l|I|^l},
\]
and so the lemma is proved.
\end{proof}
Proposition~\ref{prop:local.poly.set} follows more or less immediately from the following results, the proofs of which occupy the remainder of this section.
\begin{lemma}\label{lem:c.psi(n)}
Let $B\subset\N$, let $\nu:\N\to\R$, and suppose that $\{\nu\}(B)$ is contained within an interval $J\subset(-1/2,1/2]$ of width $\delta<1/2$. Then $\Delta_h\{\nu\}(n)=\{\Delta_h\nu\}(n)$ for every $n\in B\cap(B-h)$. Furthermore, if $\nu$ is locally polynomial of degree at most $k$ on $B$ and $\delta<2^{-k}$ then $\{\nu\}$ is also locally polynomial of degree at most $k$ on $B$.
\end{lemma}
\begin{lemma}\label{lem:psi1.psi2}
Let $B\subset\N$ and let $\nu_1,\nu_2:\N\to\R$ be functions. Suppose that $\nu_1,\nu_2$ are locally polynomial of degree at most $k_1,k_2$, respectively, on $B$. Then $\nu_1+\nu_2$ is locally polynomial of degree at most $\max\{k_1,k_2\}$ on $B$ and $\nu_1\cdot\nu_2$ is locally polynomial of degree at most $k_1+k_2$ on $B$.
\end{lemma}
\begin{proof}[Proof of Lemma~\ref{lem:c.psi(n)}]
By Lemma~\ref{lem:frac.part.difference} (i) we have
\begin{align*}
\Delta_h\{\nu\}(n) &= \{\nu(n+h)\}-\{\nu(n)\}\\
                    &= \{\nu(n+h)-\nu(n)\}\\
                    &= \{\Delta_h\nu(n)\},
\end{align*}
which was the first conclusion of the lemma. If $\delta<2^{-k}$ then Lemma~\ref{lem:frac.part.difference} (i) also implies that the image of $\{\Delta_h\nu\}$ is contained within the interval $(-2^{-k},2^{-k}]$, and by induction we may therefore assume that $\{\Delta_h\nu\}$ is locally polynomial of degree $k-1$ on $B\cap(B-h)$, and hence conclude that $\{\nu\}$ is locally polynomial of degree $k$ on $B$.
\end{proof}
\begin{proof}[Proof of Lemma~\ref{lem:psi1.psi2}]
The assertion about $\nu_1+\nu_2$ is immediate from the linearity of $\Delta_h$. To prove the assertion about $\nu_1\cdot\nu_2$, note that
\begin{align*}
\Delta_h(\nu_1\cdot\nu_2)(n) &= \nu_1(n+h)\nu_2(n+h)-\nu_1(n)\nu_2(n) \\
                               &= (\Delta_h\nu_1(n)+\nu_1(n))(\Delta_h\nu_2(n)+\nu_2(n))-\nu_1(n)\nu_2(n) \\
                               &= \Delta_h\nu_1(n)\Delta_h\nu_2(n)+\Delta_h\nu_1(n)\nu_2(n)+\nu_1(n)\Delta_h\nu_2(n),
\end{align*}
and so $\Delta_h(\nu_1\cdot\nu_2)=\Delta_h\nu_1\cdot\Delta_h\nu_2+\Delta_h\nu_1\cdot\nu_2+\nu_1\cdot\Delta_h\nu_2$. By induction each of these terms is locally polynomial of degree at most $k_1+k_2-1$ on $B\cap(B-h)$, and so $\Delta_h(\nu_1\cdot\nu_2)$ is locally polynomial of degree at most $k_1+k_2-1$ on $B\cap(B-h)$ by the first part of the lemma.
\end{proof}
\section{Strongly locally polynomial functions}\label{sec:strong.local.poly}
Let $\phi:[N]\to\R$ be a function and define $f:[N]\to\C$ by $f(n):=e(\phi(n))$. In this section we develop a criterion for $f$ to have a positive Gowers $U^k$-norm. Given a set $B\subset[N]$ of cardinality $\Omega(N)$ on which $\phi$ is locally polynomial of degree $k-1$ it is straightforward to check that $\|1_B\cdot f\|_{U^k[N]}\gg_k1$. Indeed, given that
\begin{align*}
\|1_B\cdot f\|_{U^k[N]}&\gg_k\E_{n\in[N],h\in[-N,N]^k}\left(e(\Delta_{h_1,\ldots,h_k}\phi(n))\textstyle\prod_{\omega\in\{0,1\}^k}1_B(n+\omega\cdot h)\right)\\
&=\E_{n\in[N],h\in[-N,N]^k}\left(\textstyle\prod_{\omega\in\{0,1\}^k}1_B(n+\omega\cdot h)\right)\\
&=\Prob_{n\in[N],h\in[-N,N]^k}(n+\omega\cdot h\in B\text{ for all }\omega\in\{0,1\}^k),
\end{align*}
it is an immediate consequence of the following lemma.
\begin{lemma}\label{lem:boxes.for.gcs}Suppose $B\subset[N]$ satisfies $|B|\ge\sigma N$ and for $j=0,1,2,\ldots$ write
\begin{displaymath}
B_j:=\{(n,h)\in[N]\times[-N,N]^j:n+\omega\cdot h\in B\text{ for all }\omega\in\{0,1\}^j\}.
\end{displaymath}
Then $|B_j|\gg_{j,\sigma}N^{j+1}$.
\end{lemma}
\begin{proof}The claim is trivial for the case $j=0$, and so we may proceed by induction, assuming that
\begin{equation}\label{eq:boxes.for.gcs.1}
|B_{j-1}|\gg_{j,\sigma}N^j.
\end{equation}
For each $h\in[-N,N]^{j-1}$ write $C(h):=\{n\in[N]:n+\omega\cdot h\in B\text{ for all }\omega\in\{0,1\}^{j-1}\}$ and set $r(h):=|C(h)|$; thus
\begin{equation}\label{eq:boxes.for.gcs.2}
|B_{j-1}|=\sum_{h\in[-N,N]^{j-1}}r(h).
\end{equation}
Now note simply that $B_j=\{(n,(h_1,\ldots,h_{j-1},m-n)):(h_1,\ldots,h_{j-1})\in[-N,N]^{j-1};m,n\in C(h)\}$, and so $|B_j|=\sum_{h\in[-N,N]^{j-1}}r(h)^2$, which is at least $\Omega_j\left(\left(\sum_{h\in[-N,N]^{j-1}}r(h)\right)^2/N^{j-1}\right)$ by Jensen's inequality. By (\ref{eq:boxes.for.gcs.2}) this is at least $\Omega_j(|B_{j-1}|^2/N^{j-1})$, which in turn is at least $\Omega_{j,\sigma}(N^{j+1})$ by (\ref{eq:boxes.for.gcs.1}).
\end{proof}
Our task, then, if we wish to prove that $f$ is non-uniform, is to remove the $1_B$ from the expression $\|1_B\cdot f\|_{U^k[N]}\gg_k1$. However, if we are to make use of the local behaviour of $\phi$ on $B$ then the set $B$ will have to play at least some role.

The key to reconciling this is the following.
\begin{lemma}\label{lem:application.of.gcs}
Let $\phi:[N]\to\R$ be a function and define $f:[N]\to\C$ by $f(n):=e(\phi(n))$. Then for every set $B\subset[N]$ we have
\begin{equation}\label{eq:application.of.gcs}
\|f\|_{U^k[N]}\gg_k\left|\E_{n\in[N],h\in[-N,N]^k}\left(e(\Delta_{h_1,\ldots,h_k}\phi(n))\textstyle\prod_{\omega\in\{0,1\}^k\backslash\{\mathbf0\}}1_B(n+\omega\cdot h)\right)\right|.
\end{equation}
\end{lemma}
\begin{proof}
Given a finite abelian group $G$ and functions $g_\omega:G\to\C$ indexed by $\omega\in\{0,1\}^k$, the \emph{Gowers inner product} $\langle g_\omega\rangle_{U^k(G)}$ is defined by
\begin{displaymath}
\langle g_\omega\rangle_{U^k(G)}:=\E_{x\in G,h\in G^k}\textstyle\prod_{\omega\in\{0,1\}^k}\mathcal{C}^{|\omega|}g_\omega(x+\omega\cdot h).
\end{displaymath}
Here we denote by $\mathcal{C}$ the operation of complex conjugation, and by $|\omega|$ the number of entries of $\omega$ that are equal to $1$. The \emph{Gowers--Cauchy--Schwarz inequality} \cite[(11.6)]{tao-vu} states that
\[
|\langle g_\omega\rangle_{U^k(G)}|\le\prod_{\omega\in\{0,1\}^k}\|g_\omega\|_{U^k(G)}.
\]
Since $\|1_B\cdot f\|_{U^k[N]}\le1$, setting $g_\mathbf{0}=f$ and $g_\omega=1_B\cdot f$ in this inequality, with $G=\Z/\tilde{N}\Z$ as in Definition \ref{def:[N]-norm}, yields the desired result.
\end{proof}

In order to conclude that $f$ is non-uniform it will therefore be sufficient to find a set $B$ for which we are able to place a lower bound on the right-hand side of (\ref{eq:application.of.gcs}). In this context it is natural to employ a slightly stronger notion than that of being locally polynomial.
\begin{definition}[Strongly locally polynomial]
Let $\phi:[N]\to\R$ be a function and let $B\subset[N]$. Then $\phi$ is said to be \emph{strongly locally polynomial of degree $k-1$} on $B$ if whenever $n\in[N]$ and $h\in[-N,N]^k$ satisfy $n+\omega\cdot h\in B$ for all $\omega\in\{0,1\}^k\backslash\{\mathbf0\}$ we have $\Delta_{h_1,\ldots,h_k}\phi(n)=0$.
\end{definition}
\begin{example}Suppose that $\phi(n)=\alpha n\{\beta n\}.$ Then $\phi$ is strongly locally quadratic on the set
\[
B:=\{n\in[N]:\{\beta n\}\in(-1/16,1/16)\}.
\]
Indeed, suppose that $n\in[N]$ and $h\in[-N,N]^3$ satisfy $n+\omega\cdot h\in B$ for all $\omega\in\{0,1\}^3\backslash(0,0,0)$. Then Lemma~\ref{lem:frac.part.difference} implies that $\{\beta n\}\in(-1/4,1/4)$, and so for each $\omega\in\{0,1\}^3$ (including $\omega=(0,0,0)$) the point $n+\omega\cdot h$ lies in
\[
B':=\{n\in[N]:\{\beta n\}\in(-1/4,1/4)\}.
\]
Moreover, $\phi$ is locally quadratic on $B'$ by Lemmas \ref{lem:c.psi(n)} and \ref{lem:psi1.psi2}, and so $\Delta_{h_1,h_2,h_3}\phi(n)=0$.
\end{example}
The reason for making this definition is the following result.
\begin{prop}\label{prop:strong.local.poly}
Let $\phi:[N]\to\R$ be a function and define $f:[N]\to\C$ by $f(n):=e(\phi(n))$. Suppose that $\phi$ is strongly locally polynomial of degree $k-1$ on a set $B\subset[N]$ of cardinality $\sigma N$. Then $\|f\|_{U^k[N]}\gg_{k,\sigma}1$.
\end{prop}
\begin{proof}This is a straightforward combination of Lemmas \ref{lem:boxes.for.gcs} and \ref{lem:application.of.gcs}.
\end{proof}

In light of Proposition~\ref{prop:strong.local.poly}, if we wish to prove that $\|f\|_{U^{k+1}[N]}\gg1$ then it is sufficient to find a set $B\subset[N]$ satisfying $|B|\gg N$ on which $\phi$ is strongly locally polynomial of degree at most $k$. Unfortunately, we are not able just to take $B$ to be an arbitrary set of the form given by Proposition~\ref{prop:local.poly.set}, as the following example illustrates.
\begin{example}\label{ex:bracket.overflow}
Suppose that $\phi(n)=\alpha\{\frac{1}{10}n\}$. Then Lemmas \ref{lem:c.psi(n)} and \ref{lem:psi1.psi2} imply that $\phi$ is locally linear on the set
\begin{displaymath}
B=\{n\in[N]:\{\textstyle\frac{1}{10}n\}\in(1/4,1/2]\}.
\end{displaymath}
However, if $n=6$ and $h_1=h_2=-1$ then $n+h_1$, $n+h_2$ and $n+h_1+h_2$ all lie in $B$, but $\Delta_{h_1,h_2}\phi(n)=-\alpha$. Hence $\phi$ is not \emph{strongly} locally linear on $B$.
\end{example}
Nonetheless, provided the intervals $J_\nu$ appearing in the definition of a set of the form given by Proposition~\ref{prop:local.poly.set} are sufficiently small and are sufficiently far from the boundary of $(-1/2,1/2]$, the problem exposed by Example~\ref{ex:bracket.overflow} will not occur. Before we express this precisely, let us establish some notation. For functions $\nu_1,\ldots,\nu_r:[N]\to\R$ and sets $S_1,\ldots,S_r\subset(-1/2,1/2]$ define the set
\begin{displaymath}
B^r_N(\nu_1,\ldots,\nu_r;S_1,\ldots,S_r):=\{n\in[N]:\{\nu_i(n)\}\in S_i\text{ for all }i\}.
\end{displaymath}
We typically abuse notation slightly and write $B_N$ instead of $B_N^r$. For $\varepsilon<1/2$ we write $I_\varepsilon$ for the interval $(-\varepsilon,\varepsilon)$.
\begin{lemma}\label{lem:strong.local.poly.set}Let $k\ge1$ be an integer and set $c_k:=2^{-k}(2k+1)^{-1}$. Suppose that $\phi:[N]\to\R$ is a bracket polynomial of degree at most $k$ with bracket components $\nu_1,\ldots,\nu_m$. Suppose further that
\begin{equation}\label{eq:strong.local.poly.set.3}
\delta\le c_k
\end{equation}
and that
\begin{equation}\label{eq:strong.local.poly.set.4}
\varepsilon\ge k\delta,
\end{equation}
and that $J_1,\ldots,J_m$ are intervals of width at most $\delta$ inside $I_{\frac{1}{2}-\varepsilon}$.

Then there is a set $B\subset[N]$ on which $\phi$ is locally polynomial of degree at most $k$ and such that if $n\in[N]$, $h\in[-N,N]^{k+1}$ and $n+\omega\cdot h\in B_N(\nu_1,\ldots,\nu_m;J_1,\ldots,J_m)$ for every $\omega\in\{0,1\}^{k+1}\backslash\{\mathbf0\}$ then $n+\omega\cdot h\in B$ for every $\omega\in\{0,1\}^{k+1}$. In particular, $\phi$ is strongly locally polynomial of degree at most $k$ on the set $B_N(\nu_1,\ldots,\nu_m;J_1,\ldots,J_m)$.
\end{lemma}
Before we prove Lemma \ref{lem:strong.local.poly.set}, let us note that in combination with Proposition \ref{prop:strong.local.poly} it immediately implies the following result.
\begin{prop}\label{prop:recur.implies.non-uniform}Let $k\ge1$ be an integer and let $c_k$ be as in Lemma \ref{lem:strong.local.poly.set}. Suppose that $\phi:[N]\to\R$ is a bracket polynomial of degree at most $k$ with bracket components $\nu_1,\ldots,\nu_m$. Suppose further that $\delta\le c_k$ and that $\varepsilon\ge k\delta$, and that $J_1,\ldots,J_m$ are intervals of width at most $\delta$ inside $I_{\frac{1}{2}-\varepsilon}$ such that
\[
|B_N(\nu_1,\ldots,\nu_m;J_1,\ldots,J_m)|\ge\sigma N.
\]
Then, defining $f:[N]\to\C$ by $f(n):=e(\phi(n))$, we have $\|f\|_{U^{k+1}[N]}\gg_{k,\sigma}1$.
\end{prop}
\begin{proof}[Proof of Lemma \ref{lem:strong.local.poly.set}]The lemma is trivial for genuine polynomials with $B=[N]$, and so we may assume that we are in one of three cases:
\\
\\
\textit{Case 1.} $\phi=\{\theta\}$ with $\theta$ of degree at most $k$.
\\
\textit{Case 2.} $\phi=\theta_1+\theta_2$ with $\theta_i$ of degree at most $k_i$ and $k_1,k_2\le k$.
\\
\textit{Case 3.} $\phi=\theta_1\cdot\theta_2$ with $\theta_i$ of degree at most $k_i$ and $k_1+k_2\le k$.
\\
\\
We proceed by induction on the number of operations required to construct $\phi$. Therefore, in case 1 we may assume that there exists a set $B_0\subset[N]$ on which $\theta$ is locally polynomial; in cases 2 and 3 we may assume that there exist a set $B_1\subset[N]$ on which $\theta_1$ is locally polynomial and a set $B_2\subset[N]$ on which $\theta_2$ is locally polynomial; and in each case we may assume that $n+\omega\cdot h\in B_i$ for every $\omega\in\{0,1\}^{k+1}$ whenever $n+\omega\cdot h\in B_N(\nu_1,\ldots,\nu_m;J_1,\ldots,J_m)$ for every $\omega\ne\mathbf0$.

\begin{sloppypar}In case 1 we may assume that $\phi=\nu_m=\{\theta\}$. Suppose that $n+\omega\cdot h\in B_N(\nu_1,\ldots,\nu_m;J_1,\ldots,J_m)$ for every $\omega\ne\mathbf0$. Then $\Delta_h\theta(n)=0$ by the inductive hypotheses, and so repeated application of Lemma \ref{lem:frac.part.difference} implies that $\{\theta(n)\}\in J_m':=J_m+(-k\delta,k\delta)$. The inequality (\ref{eq:strong.local.poly.set.4}) then implies that $J_m'\subset(-1/2,1/2]$, whilst (\ref{eq:strong.local.poly.set.3}) and the definition of $c_k$ imply that $|J_m'|\le2^{-k}$. Lemma~\ref{lem:c.psi(n)} therefore implies that we may take $B=B_0\cap B_N(\theta;J_m')$.\end{sloppypar}

In cases 2 and 3 we may simply take $B=B_1\cap B_2$ by Lemma~\ref{lem:psi1.psi2}. 
\end{proof}
In the event that the bracket components $\nu_i$ appearing in Proposition \ref{prop:recur.implies.non-uniform} are linear it is elementary to show that there exist intervals $J_i$ satisfying the hypotheses of that proposition, with $\sigma$ depending only on $k$ and $m$. Indeed, one can even insist that the intervals $J_i$ be centred at zero, as follows.
\begin{lemma}[Sets of linear bracket polynomials are strongly recurrent]\label{lem:linear.Bohr.set.size.bound.lower}
Let $\alpha_1,\ldots,\alpha_r\in\R$, let $\delta>0$ and let $\nu_i(n):=\{\alpha_in\}$. Then
\[
|B_N(\nu_1,\ldots,\nu_r;I_\delta,\ldots,I_\delta)|\gg_{r,\delta}N.
\]
\end{lemma}
\begin{proof}
A pigeonholing argument similar to that used in \cite[Lemma 4.20]{tao-vu} gives points $\xi_1,\ldots,\xi_r\in\R$ and a subset $A\subset[N]$ satisfying $|A|\gg_{r,\delta}N$ such that whenever $n\in A$ we have $\|\alpha_jn-\xi_j\|_{\R/\Z}<\delta/2$ for every $j=1,\ldots,r$.

Following \cite[Lemma 4.20]{tao-vu}, note that if $\|\alpha_jn-\xi_j\|_{\R/\Z}<\delta/2$ and $\|\alpha_jn'-\xi_j\|_{\R/\Z}<\delta/2$ for every $j=1,\ldots,r$ then by the triangle inequality we have $\|\alpha_j(n-n')\|_{\R/\Z}<\delta$ for every $j=1,\ldots,r$. Therefore, writing $m$ for the maximum element of $A$, the set $(m-A)\backslash\{0\}$ is contained in $B_N(\nu_1,\ldots,\nu_r;I_\delta,\ldots,I_\delta)$, and so
\[
|B_N(\nu_1,\ldots,\nu_r;I_\delta,\ldots,I_\delta)|\ge|m-A|-1\gg_{r,\delta}N.
\]
\end{proof}
Combined with Proposition \ref{prop:recur.implies.non-uniform}, this immediately implies Theorem \ref{thm:bracket.non.uniform} in the case that every bracket component of $\Phi$ is linear. To conclude Theorem \ref{thm:bracket.non.uniform} in general, we require the following generalisation of Lemma \ref{lem:linear.Bohr.set.size.bound.lower}.
\begin{theorem}[Recurrence of bracket polynomials; precise statement]\label{lem:general.Bohr.set.size.bound.lower}
Let $\Theta_1,\ldots,\Theta_r$ be constant-free bracket forms and suppose that $\theta_1,\ldots,\theta_r$ are realisations of $\Theta_1,\ldots,\Theta_r$, respectively. Let $\delta>0$. Then, provided $N$ is sufficiently large in terms of $\Theta_1,\ldots,\Theta_r$ and $\delta$, we have
\[
|B_N(\theta_1,\ldots,\theta_r;I_\delta,\ldots,I_\delta)|\gg_{\Theta_1,\ldots,\Theta_r,\delta}N.
\]
\end{theorem}
\begin{remark}\label{rem:const.free}
The restriction here to const\-ant-free bracket forms is necessary. For example, in the case $r=1$, if $\Theta_1$ were the bracket form $\{1/2+\alpha n\}$ then for $c\ll N^{-1}$ the realisation $\{1/2+cn\}$ of $\Theta_1$ would not satisfy the proposition.
\end{remark}
We prove Theorem \ref{lem:general.Bohr.set.size.bound.lower} in Section \ref{sec:literature}.

Compared to Lemma \ref{lem:linear.Bohr.set.size.bound.lower}, the proof of which was very straightforward, Theorem \ref{lem:general.Bohr.set.size.bound.lower} appears to be rather deep, in that our proof makes use of two major results from the literature. The first is work of Bergelson and Leibman \cite{berg.leib} that allows us to express a bracket polynomial in terms of a so-called \emph{polynomial sequence} on a nilmanifold. The second is a difficult theorem of Green and Tao \cite{poly.seq} describing the distribution of such polynomial sequences.

Of course, it may well be that there is an elementary proof of Theorem \ref{lem:general.Bohr.set.size.bound.lower}, or at least of some variant of it that is still strong enough to imply Theorem \ref{thm:bracket.non.uniform}. In Sections \ref{sec:weak.recur} and \ref{sec:approx.loc.poly} we give elementary arguments establishing weak versions of Theorem \ref{lem:general.Bohr.set.size.bound.lower} that are sufficient to prove Theorem \ref{thm:bracket.non.uniform} in certain simple cases.
\section{Bases and coordinates on nilmanifolds}
Our aim now is to prove Theorem \ref{lem:general.Bohr.set.size.bound.lower}. As we remarked at the end of the last section, the proof makes use of results of Bergelson and Leibman \cite{berg.leib} that allow us to express a bracket polynomial in terms of a so-called \emph{polynomial sequence} on a nilmanifold, and results of Green and Tao \cite{poly.seq} that describe the behaviour of such a sequence.

Even just to state these results requires a fair amount of background and notation concerning nilmanifolds, which we introduce in this section. This allows us to state the results of Bergelson--Leibman and Green--Tao in the next section, where we also prove Theorem \ref{lem:general.Bohr.set.size.bound.lower}.

At this point let us recall our convention, which applies throughout this paper, that when we write that $G/\Gamma$ is a nilmanifold we assume that $G$ is a \emph{connected} and \emph{simply connected} nilpotent Lie group. This is consistent with a standing assumption in \cite{poly.seq}, for example.

We start this section by introducing Mal'cev bases and coordinates on nilmanifolds. These are standard concepts in the study of nilpotent Lie groups and nilmanifolds, and are well documented in the literature; the reader may consult \cite{cor-gre}, for example, for more detailed background.

\begin{definition}[Mal'cev basis of a nilpotent Lie algebra {\cite[\S1.1.13]{cor-gre}}]\label{def:malc.bas.alg}
Let $\g$ be an $m$-dimensional nilpotent Lie algebra, and let $\mathcal{X}=\{X_1,\ldots,X_m\}$ be a basis for $\g$ over $\R$. Then $\mathcal{X}$ is said to be a \emph{Mal'cev basis} for $\g$ if for each $j=0,\ldots,m$ the subspace $\h_j$ spanned by the vectors $X_{j+1},\ldots,X_m$ is a Lie algebra ideal in $\g$.

In the event that $\g$ is the Lie algebra of a connected, simply connected nilpotent Lie group $G$, we sometimes say that $\mathcal{X}$ is a Mal'cev basis \emph{for $G$}.
\end{definition}
\begin{remark}\label{rem:mal.exist}It follows from \cite[Theorem 1.1.13]{cor-gre} that every nilpotent Lie algebra admits a Mal'cev basis. We will not need this general fact in this paper, however, since we deal only with explicit bases that can easily be verified to be Mal'cev bases.
\end{remark}
We can use a Mal'cev basis for a connected, simply connected nilpotent Lie group $G$ to place a coordinate system on $G$, using the following result.
\begin{prop}
Let $G$ be a connected, simply connected, $m$-dimensional nilpotent Lie group with Lie algebra $\g$. Then for every $g\in G$ there is a unique $m$-tuple $(t_1,\ldots,t_m)\in\R^m$ such that
\begin{equation}\label{eq:mal.coord}
g=\exp(t_1X_1)\cdots\exp(t_mX_m).
\end{equation}
\end{prop}
\begin{proof}It follows from \cite[Theorem 1.2.1 (a)]{cor-gre} and \cite[Proposition 1.2.7 (c)]{cor-gre} that every element in $\exp\g$ can be expressed uniquely in the form (\ref{eq:mal.coord}), and from \cite[Theorem 1.2.1 (a)]{cor-gre} that $\exp\g=G$.
\end{proof}
This allows us to make the following definition.
\begin{definition}[Mal'cev coordinates]
Let $G$ be a connected, simply connected, $m$-dimensional nilpotent Lie group with Lie algebra $\g$, and let $g\in G$. Then we call the $t_i$ appearing in the expression (\ref{eq:mal.coord}) the \emph{Mal'cev coordinates} of $g$. We define the \emph{Mal'cev coordinate map} $\psi=\psi_{\mathcal{X}}:G\to\R^m$ by
\[
\psi(g):=(t_1,\ldots,t_r).
\]
\end{definition}
\begin{definition}[Mal'cev basis for a nilmanifold]Let $G/\Gamma$ be an $m$-dimensional nilmanifold. Then a Mal'cev basis $\mathcal{X}$ for $G$ is said to be \emph{compatible with $\Gamma$} if $\Gamma$ consists precisely of those elements whose Mal'cev coordinates are all integers. We also indicate this by saying simply that $\mathcal{X}$ is a Mal'cev basis \emph{for $G/\Gamma$}.

In the event that $\mathcal{X}$ is a Mal'cev basis for $G/\Gamma$, for each $g\in G$ there is a unique $z\in\Gamma$ for which the coordinates of $gz$ all lie in $(-1/2,1/2]$; see, for example, \cite[Lemma A.14]{poly.seq}. In this case, we call the coordinates of $gz$ the \emph{nilmanifold coordinates} of $g$, and define the \emph{nilmanifold coordinate map} $\chi=\chi_\mathcal{X}:G\to(-1/2,1/2]^m$ by
\[
\chi(g):=\psi(gz).
\] 
\end{definition}
These definitions are somewhat technical, so at this point the reader may find it instructive to consider the following example.
\begin{example}\label{eg:heis.bases}
Let $G/\Gamma$ be the Heisenberg nilmanifold. Let
\[
X_1=\log\left(\begin{smallmatrix}1&0&0\\0&1&1\\0&0&1\end{smallmatrix}\right)\qquad\qquad
X_2=\log\left(\begin{smallmatrix}1&1&0\\0&1&0\\0&0&1\end{smallmatrix}\right)\qquad\qquad
X_3=\log\left(\begin{smallmatrix}1&0&1\\0&1&0\\0&0&1\end{smallmatrix}\right)
\]
and let $\mathcal X=\{X_1,X_2,X_3\}$. It is straightforward to check that
\[
\left(\begin{smallmatrix}1&x&z\\0&1&y\\0&0&1\end{smallmatrix}\right)=\exp(yX_1)\exp(xX_2)\exp(zX_3),
\]
and so $\mathcal X$ is a Mal'cev basis for $G/\Gamma$ and its Mal'cev coordinate map $\psi_\mathcal{X}$ satisfies
\[
\psi_\mathcal{X}\left(\left(\begin{smallmatrix}1&x&z\\0&1&y\\0&0&1\end{smallmatrix}\right)\right)=(y,x,z).
\]
On the other hand, if we change the order of $X_1,X_2,X_3$, setting
\[
Y_1=X_2\qquad\qquad Y_2=X_1\qquad\qquad Y_3=X_3
\]
and setting $\mathcal Y=\{Y_1,Y_2,Y_3\}$, then we have
\[
\left(\begin{smallmatrix}1&x&z\\0&1&y\\0&0&1\end{smallmatrix}\right)=\exp(xY_1)\exp(yY_2)\exp((z-xy)Y_3).
\]
Thus $\mathcal Y$ is also a Mal'cev basis for $G/\Gamma$, but the Mal'cev coordinate map $\psi_\mathcal{Y}$ satisfies
\[
\psi_\mathcal{Y}\left(\left(\begin{smallmatrix}1&x&z\\0&1&y\\0&0&1\end{smallmatrix}\right)\right)=(x,y,z-xy).
\]
In particular, note that changing the order of the basis elements does not simply change the order of the Mal'cev coordinates.
\end{example}

For the purposes of this paper we will need to consider slightly more specific Mal'cev bases than those we have defined so far.
\begin{definition}[Filtration of a nilpotent group]
Let $G$ be a nilpotent group. A \emph{filtration} $G_\bullet$ of $G$ is a sequence of closed connected subgroups
\[
G=G_0=G_1\supset G_2\supset\cdots\supset G_d\supset G_{d+1}=\{1\}
\]
with the property that $[G_i,G_j]\subset G_{i+j}$ for all integers $i,j\ge0$. We define the \emph{degree} of $G_\bullet$ to be the minimal integer $d$ such that $G_{d+1}=\{1\}$.
\end{definition}
For example, the lower central series is a filtration with degree equal to the nilpotency class of the group. \begin{definition}[Mal'cev basis adapted to a filtration {\cite[Definition 2.1]{poly.seq}}]
Let $G/\Gamma$ be an $m$-dimensional nilmanifold and let $G_\bullet$ be a filtration for $G$. A Mal'cev basis $\mathcal{X}=\{X_1,\ldots,X_m\}$ for $G/\Gamma$ is said to be \emph{adapted to $G_\bullet$} if for each $i=1,\ldots,d$ we have $G_i=\exp\h_{m-\dim G_i}$. Here $\h_j=\Span\{X_{j+1},\ldots,X_m\}$ as in Definition \ref{def:malc.bas.alg}.
\end{definition}
\begin{remarks}
According to a result of Mal'cev \cite{malcev}, every nilmanifold admits a Mal'cev basis adapted to the lower central series. It is easy to check that the Mal'cev bases for the Heisenberg group that we considered in Example \ref{eg:heis.bases} are both Mal'cev bases adapted to the lower central series.
\end{remarks}
We close this section by introducing some higher-step variants of the Heisenberg nilmanifold, and some Mal'cev bases for them. We denote by $T_p$ the group of $(p+1)\times(p+1)$ real upper-triangular matrices with every diagonal element equal to 1; thus, for example, $T_2$ is the Heisenberg group. Define $Z_p$ to be the subgroup of $T_p$ consisting of those matrices having only integer entries. The quotient $T_p/Z_p$ is then a $p$-step nilmanifold. We define a Mal'cev basis for $T_p/Z_p$ as follows.
\begin{definition}[Standard basis for an upper-triangular nilmanifold]
Let $p\in\N$, and let $U$ be the set of all elements of $T_p$ that have one non-diagonal entry equal to 1, and every other non-diagonal entry equal to zero. Now for $i=1,\ldots,p$ define $U_i$ to be the set of elements of $U$ in which the unique non-diagonal non-zero entry is at a distance $i$ from the main diagonal; more precisely, if the unique non-diagonal non-zero entry of $A\in U$ is the $(j,k)$ entry then $A$ belongs to $U_{k-j}$. Note, therefore, that $U$ is the union of the $U_i$, and that for each $l=1,\ldots,p+1$ and each $i=1,\ldots,p$ the set $U_i$ contains at most one element with a non-diagonal non-zero entry in row $l$.

Then we define the \emph{standard basis} $\mathcal{X}_p$ for $T_p/Z_p$ to consist of the elements of $U$, ordered such that if $i<j$ then every element of $U_i$ appears before every element of $U_j$, and such that if $j<k$ then an element of $U_i$ whose non-diagonal non-zero entry lies in row $j$ appears before any element of $U_i$ whose non-diagonal non-zero entry lies in row $k$.

More generally, let $r\in\N$, and for each $i=1,\ldots,p$ and each $j=1,\ldots,r$ define the subset $U_{i,j}$ of the direct product $T_p^r$ to be the set
\[
\{(A_1,\ldots,A_r)\in T_p^r:A_j\in U_i; A_k=1\text{ for all }k\ne j\}.
\]
Then we define the \emph{standard basis} $\mathcal{X}_{p,r}$ for $T_p^r/Z_p^r$ to consist of those elements belonging to the union of the $U_{i,j}$, ordered such that if $i<i'$ and $j,j'$ are arbitrary then every element of $U_{i,j}$ appears before every element of $U_{i',j'}$; such that if $j<j'$ and $i$ is arbitrary then every element of $U_{i,j}$ appears before every element of $U_{i,j'}$; and such that if $A,B\in U_i$ and $A$ appears before $B$ in the basis $\mathcal{X}_p$ then the element $(1,\ldots,1,A,1,\ldots,1)$ of $U_{i,j}$ appears before the element $(1,\ldots,1,B,1,\ldots,1)$ of $U_{i,j}$ in $\mathcal{X}_{p,r}$.
\end{definition}
Thus, for example, the standard basis for $T_2^2$ consists of the elements
\[
\left(\left(\begin{smallmatrix}1&1&0\\0&1&0\\0&0&1\end{smallmatrix}\right),\left(\begin{smallmatrix}1&0&0\\0&1&0\\0&0&1\end{smallmatrix}\right)\right),
\]
\[
\left(\left(\begin{smallmatrix}1&0&0\\0&1&1\\0&0&1\end{smallmatrix}\right),\left(\begin{smallmatrix}1&0&0\\0&1&0\\0&0&1\end{smallmatrix}\right)\right),
\]
\[
\left(\left(\begin{smallmatrix}1&0&0\\0&1&0\\0&0&1\end{smallmatrix}\right),\left(\begin{smallmatrix}1&1&0\\0&1&0\\0&0&1\end{smallmatrix}\right)\right),
\]
\[
\left(\left(\begin{smallmatrix}1&0&0\\0&1&0\\0&0&1\end{smallmatrix}\right),\left(\begin{smallmatrix}1&0&0\\0&1&1\\0&0&1\end{smallmatrix}\right)\right),
\]
\[
\left(\left(\begin{smallmatrix}1&0&1\\0&1&0\\0&0&1\end{smallmatrix}\right),\left(\begin{smallmatrix}1&0&0\\0&1&0\\0&0&1\end{smallmatrix}\right)\right),
\]
\[
\left(\left(\begin{smallmatrix}1&0&0\\0&1&0\\0&0&1\end{smallmatrix}\right),\left(\begin{smallmatrix}1&0&1\\0&1&0\\0&0&1\end{smallmatrix}\right)\right),
\]
in that order.
\begin{remark}\label{rem:std.bas.is.mal}It is straightforward to check that the standard basis for $T_p/Z_p$ is indeed a Mal'cev basis, adapted to the lower central series \cite[\S5]{berg.leib}.
\end{remark}

The definition of a Mal'cev basis $\mathcal{X}$ of a nilpotent Lie algebra $\g$ requires that the vector subspaces $\h_j$ are Lie algebra ideals, which is to say that
\[
[\g,\h_j]\subset\h_j\qquad\qquad\qquad(j=0,\ldots,m-1).
\]
If $\mathcal{X}$ is a Mal'cev basis for a nilmanifold adapted to a filtration, however, then it obeys the stronger property that
\begin{equation}\label{eq:nesting}
[\g,\h_j]\subset\h_{j+1}\qquad\qquad\qquad(j=0,\ldots,m-1);
\end{equation}
here, we adopt the convention that $\h_m=\{0\}$. In \cite{poly.seq}, (\ref{eq:nesting}) is called the \emph{nesting property}. The nesting property turns out to be an important technical condition for various results from \cite[Appendix A]{poly.seq} that we use repeatedly in this paper. However, when we apply these results it is not always the case that we are applying them to a Mal'cev basis adapted to a filtration. It is therefore useful to introduce the following definition.
\begin{definition}[Nested Mal'cev basis]
Let $\g$ be an $m$-dimensional nilpotent Lie algebra. A Mal'cev basis for $\g$ that satisfies (\ref{eq:nesting}) is called a \emph{nested} Mal'cev basis.
\end{definition}

We close this section by defining what it means for a basis for $\g$ to be \emph{rational} in a quantitative sense, which is an essential concept for understanding the results of Green and Tao that we present in the next section.
\begin{definition}[Quantitative rationality]
The \emph{height} of a rational number $x$ is defined to be $\max\{|a|,|b|\}$ if $x=a/b$ in reduced form. If $Y,X_1,\ldots,X_r\in\R^m$ then we say that $Y$ is a \emph{$Q$-rational combination} of the $X_i$ if there are rationals $q_i$ of height at most $Q$ such that $Y=q_1X_1+\ldots+q_rX_r$.
\end{definition}
\begin{definition}[Rationality of a basis]
Let $\g$ be a nilpotent Lie algebra with basis $\mathcal{X}=\{X_1,\ldots,X_m\}$. We say that $\mathcal{X}$ is \emph{$Q$-rational} if the structure constants $c_{ijk}$ appearing in the relations
\[
[X_i,X_j]=\sum_kc_{ijk}X_k
\]
are all rational of height at most $Q$.
\end{definition}
\section{Polynomial sequences on nilmanifolds}\label{sec:literature}
In this section we introduce results of Bergelson and Leibman \cite{berg.leib} and Green and Tao \cite{poly.seq} that allow us to prove Theorem \ref{lem:general.Bohr.set.size.bound.lower}. We start by describing the work of Bergelson and Leibman.

Recall from the introduction that the bracket polynomial $\{\alpha n[\beta n]\}$ arises naturally from the sequence
\begin{equation}\label{eq:lin.poly.form}
g(n)=\left(\begin{array}{ccc}
               1 & -\alpha n & 0 \\
               0 & 1        & \beta n \\
               0 & 0        & 1
               \end{array}\right)
\end{equation}
in the Heisenberg nilmanifold. Remarkably, Bergelson and Leibman show that \emph{every} bracket polynomial arises in a similar way.
\begin{definition}[Polynomial mappings and polynomial forms]
A map $\rho:\Z\to T_p^r$ is said to be a \emph{polynomial mapping} of degree at most $k$ if there are polynomials $\phi_{i,j,l}$ of degree at most $k$ such that for each $n\in\Z$ the $(i,j)$-entry of the matrix $\rho_l(n)$ is given by $\phi_{i,j,l}(n)$. If the $\phi_{i,j,l}$ all have zero constant term then $\rho$ is said to be a \emph{constant-free} polynomial mapping.

Now let $\Phi_{i,j,l}$ be polynomial forms of degree at most $k$ in the sense of Definition \ref{def:poly.form}. Then the $r$-tuple $P=(P_1,\ldots,P_r)$ of $(p+1)\times(p+1)$ matrices whose diagonal entries are all 1, and such that every above-diagonal $(i,j)$-entry of $P_l$ is equal to the polynomial form $\Phi_{i,j,l}$, is said to be a \emph{polynomial form} of degree at most $k$ on $T_p^r$. If every $\Phi_{i,j,l}$ is constant free then $P$ is also said to be \emph{constant free}.

Finally, for each $i,j,l$ let $\phi_{i,j,l}$ be a realisation of $\Phi_{i,j,l}$. Let $\rho:\Z\to T_p^r$ be the polynomial mapping defined by setting the $(i,j)$-entry of $\rho_l(n)$ equal to $\phi_{i,j,l}(n)$. Then $\rho$ is said to be a \emph{realisation} of the polynomial form $P$.
\end{definition}
\begin{remarks}If the $\alpha$ and $\beta$ appearing in (\ref{eq:lin.poly.form}) are taken to be elements of some alphabet (as opposed to real numbers) then the matrix $g(n)$ can be viewed as a polynomial form of degree at most $1$ on $T_2$. It follows from \cite[\S5.8]{berg.leib} that for every $p,r$ the nilmanifold coordinates $\chi(g)$ of an element $g\in T_p^r$ with respect to the standard basis are given by bracket expressions in the entries of the matrices appearing in $g$. Thus, in particular, if $P$ is a polynomial form on $T_p^r$ then each coordinate $\chi(P)_i$ naturally defines a bracket form.
\end{remarks}
\begin{theorem}[Bergelson--Leibman \cite{berg.leib}]\label{thm:berg.leib}
Let $\Theta_1,\ldots,\Theta_r$ be constant-free bracket forms. Then there exist $p\ge1$, a constant-free polynomial form $P$ on $T_p^r$, and a nested Mal'cev basis $\mathcal{Y}=\{Y_1,\ldots,Y_m\}$ for $T_p^r/Z_p^r$ such that each element $Y_i$ of $\mathcal{Y}$ is equal to either an element $X_j$ of the standard basis $\mathcal{X}$ or its inverse $-X_j$, and such that for every $i=1,\ldots,r$ we have either $\{\Theta_i\}=\chi_\mathcal{Y}(P)_{m-r+i}$ or $\{-\Theta_i\}=\chi_\mathcal{Y}(P)_{m-r+i}$.
\end{theorem}
Theorem \ref{thm:berg.leib} is not stated exactly in this way in Bergelson and Leibman's paper, but it can be read out of the work contained therein. In particular, Bergelson and Leibman express concrete bracket polynomials in terms of concrete polynomial mappings, whereas Theorem \ref{thm:berg.leib} expresses bracket forms in terms of polynomial forms. The reason for this modification is to make it clear that all implied constants appearing in our subsequent work are uniform across all realisations of a given bracket form. In Appendix \ref{sec:berg.leib} we offer a brief discussion of how to obtain Theorem \ref{thm:berg.leib} from the original work of Bergelson and Leibman.

It turns out to be useful to note, as we do in Lemma \ref{lem:finer.filt} below, that polynomial mappings into nilmanifolds are examples of slightly more specific objects called \emph{polynomial sequences} on nilmanifolds. We define these now.

\begin{definition}[Polynomial sequence in a nilpotent group]\label{def:poly.seq}
Let $G$ be a nilpotent group with a filtration $G_\bullet$ and let $g:\Z\to G$ be a sequence. For $h\in\Z$ define $\partial_hg(n):=g(n+h)g(n)^{-1}$. Then $g$ is said to be a \emph{polynomial sequence} with respect to the filtration $G_\bullet$ if $\partial_{h_i}\ldots\partial_{h_1}g$ takes values in $G_i$ for all $i\in\N$ and $h_1,\ldots,h_i\in\Z$.
\end{definition}
\begin{example}\label{ex:poly.seq}
In the Heisenberg group $T_2$, the sequence
\[
\left(\begin{array}{ccc}
1 & \alpha_1n & \beta n^2\\
0 & 1         & \alpha_2n\\
0 & 0         & 1
\end{array}\right)
\]
is a polynomial sequence with respect to the lower central series. In the group $T_3$ the sequence
\[
\left(\begin{array}{cccc}
1 & \alpha_1n & \beta_1n^2 & \gamma n^3\\
0 & 1         & \alpha_2n  & \beta_2n^2\\
0 & 0         & 1          & \alpha_3n\\
0 & 0         & 0          & 1
\end{array}\right)
\]
is a polynomial sequence with respect to the lower central series. More generally, if $g(n)$ is a sequence inside the group $T_p$ defined by a matrix, each of whose entries is a polynomial in $n$ of degree at most its distance from the main diagonal, then $g$ is a polynomial sequence with respect to the lower central series. We leave it to the reader to verify this fact.
\end{example}
The polynomial sequences given in Example \ref{ex:poly.seq} are of course also polynomial mappings into $T_p$. However, not every polynomial mapping is a polynomial sequence with respect to the lower central series, as can be seen by considering, for example, the mapping
\[
\left(\begin{array}{cc}
1 & \alpha n^2 \\
0 & 1
\end{array}\right)
\]
into $T_1$. It turns out, however, that every polynomial mapping into $T_p^r$ is a polynomial sequence with respect to \emph{some} filtration.
\begin{lemma}\label{lem:finer.filt}
Let $k,p,r\in\N$. Then there is a filtration $G_\bullet$ of $T_p^r$ of degree at most $O_{k,p}(1)$ with respect to which every polynomial mapping $\rho:\Z\to T_p^r$ of degree at most $k$ is a polynomial sequence, and such that the standard basis $\mathcal{X}$ for $T_p^r/Z_p^r$ is a Mal'cev basis adapted to $G_\bullet$.
\end{lemma}
We prove Lemma \ref{lem:finer.filt} shortly, but first we note the following statement, which is a key ingredient of Lemma \ref{lem:finer.filt}.
\begin{lemma}\label{lem:plane.home}
Let $k,p,r\in\N$. Then there is a some $d\in\N$ depending only on $k$ and $p$ such that if $\rho$ is an arbitrary polynomial mapping of degree at most $k$ into $T_p^r$ then the derivatives $\partial_{h_{d+1}}\ldots\partial_{h_1}\rho$ are all trivial.
\end{lemma}
The proof of Lemma \ref{lem:plane.home} is a straightforward exercise, but given its importance to this paper we present it in full in Appendix \ref{ap:poly.map}.
\begin{proof}[Proof of Lemma \ref{lem:finer.filt}]
Let $G'_\bullet$ be the lower central series of $T_p^r$, which is a filtration of degree $p$, and let $d$ be the natural number given by Lemma \ref{lem:plane.home}. Following the procedure outlined in the paragraphs following \cite[Corollary 6.8]{poly.seq}, define a finer filtration $G_\bullet$ of degree $pd$ by setting $G_i=G'_{\lceil i/d\rceil}$. Then $\rho$ is a polynomial sequence for the filtration $G_\bullet$, and so the first conclusion of the lemma is proved.

The fact that $\mathcal{X}$ is a Mal'cev basis adapted to $G_\bullet$ follows straightforwardly from the fact (noted in Remark \ref{rem:std.bas.is.mal}) that it is a Mal'cev basis adapted to $G'_\bullet$, and from the fact that each $G_i$ is equal to some $G'_j$.
\end{proof}
Our objective in this section is to prove Theorem \ref{lem:general.Bohr.set.size.bound.lower}, which is a recurrence result for bracket polynomials, modulo 1. Moreover, in light of Theorem \ref{thm:berg.leib} and Lemma \ref{lem:finer.filt} the study of bracket polynomials reduces, in a sense, to the study of polynomial sequences on nilmanifolds. This suggests that it would be useful to understand the distribution of such polynomial sequences. We now describe deep work of Green and Tao investigating precisely this.

We start by defining a metric on a nilmanifold.
\begin{definition}[Metrics on nilmanifolds]Given a nilmanifold $G/\Gamma$ with a rational nested Mal'cev basis $\mathcal{X}$, we define a metric $d=d_{G,\mathcal{X}}$ on $G$ by taking the largest metric such that $d(x,y)\le|\psi(xy^{-1})|$ for all $x,y\in G$; here, and throughout this paper, $|\,\cdot\,|$ denotes the $\ell^\infty$-norm on $\R^{\dim G}$. We also define a metric $d=d_{G/\Gamma,\mathcal{X}}$ on $G/\Gamma$ by
\[
d(x\Gamma,y\Gamma)=\inf_{z\in\Gamma}d(x,yz).
\]
\end{definition}
\begin{remark}
It is shown in \cite[Lemma A.15]{poly.seq} that this is a metric on $G/\Gamma$. Note that, although the hypotheses of that lemma include the assumption that $\mathcal{X}$ is a Mal'cev basis adapted to some filtration, all that is used is that the elements of $\Gamma$ have integer coordinates, and that $\mathcal{X}$ is rational and nested (the assumption that $\mathcal{X}$ is nested being necessary in order to apply \cite[Lemmas A.4 and A.5]{poly.seq}).
\end{remark}
This definition is rather abstract. However, we never need to calculate it explicitly, and the only properties we require are detailed in \cite[Appendix A]{poly.seq}. The interested reader may find a more explicit formulation of this metric in \cite[Definition 2.2]{poly.seq}.

One immediate property of the metric $d$ is that it is \emph{right invariant}, in the sense that
\begin{equation}\label{eq:right.invar}
d(x,y)=d(xg,yg)
\end{equation}
for every $x,y,g\in G$. Another is that the metric $d$ is symmetric at the identity, in the sense that
\begin{equation}\label{eq:d.sym}
d(x,1)=d(x^{-1},1)
\end{equation}
for every $x\in G$.

Once we have a metric on $G/\Gamma$ we are able to make the following definitions.
\begin{definition}[Lipschitz norm]
Define the \emph{Lipschitz norm} $\|\,\cdot\,\|_\Lip$ on the space of Lipschitz functions $f:G/\Gamma\to\C$ by
\[
\|f\|_\Lip:=\|f\|_\infty+\sup_{x\ne y}\frac{|f(x)-f(y)|}{d(x,y)},
\]
\end{definition}
\begin{definition}[Equidistribution]\label{def:equidist}
Let $G/\Gamma$ be a nilmanifold, and write $\mu$ for the unique normalised Haar measure on $G/\Gamma$. A sequence $(g(n)\Gamma)_{n\in\Z}$ is said to be \emph{equidistributed} in $G/\Gamma$ if for every continuous function $f:G/\Gamma\to\C$ we have
\begin{displaymath}
\E_{n\in\Z}f(g(n)\Gamma)=\int_{G/\Gamma}fd\mu.
\end{displaymath}
Now let $\delta>0$ be a parameter and let $Q\subset\Z$ be an arithmetic progression of length $N$. A sequence $(g(n)\Gamma)_{n\in Q}$ is said to be \emph{$\delta$-equidistributed} in $G/\Gamma$ if for every Lipschitz function $f:G/\Gamma\to\C$ we have
\begin{displaymath}
\left|\E_{n\in Q}f(g(n)\Gamma)-\int_{G/\Gamma}fd\mu\right|\le\delta\|f\|_\Lip.
\end{displaymath}
We say that $(g(n)\Gamma)_{n\in Q}$ is \emph{totally $\delta$-equidistributed} if $(g(n)\Gamma)_{n\in Q'}$ is $\delta$-equidistributed for every subprogression $Q'\subset Q$ of length at least $\delta N$.
\end{definition}
The key result of Green and Tao shows that \emph{every} polynomial sequence $g$ on a nilmanifold $G/\Gamma$ has an `equidistributed' component, in a certain precise sense. More specifically, their result allows us to factor an arbitrary polynomial sequence $g$ on a nilmanifold $G/\Gamma$ as a product $\varepsilon g'\gamma$, in which $g'$ is $\delta$-equidistributed on some subnilmanifold $G'/\Gamma'$ of $G$, in which $\varepsilon$ is `almost constant' in a certain sense, and in which $\gamma$ is periodic with fairly short period. Thus, ignoring for the moment the effects of the `almost constant' sequence $\varepsilon$, we see that $g$ is roughly equidistributed on the union of a small number of translates of the subnilmanifold $G'/\Gamma'$.

For this to make sense, we must first define what we mean by a `subnilmanifold'.
\begin{definition}[Rational subgroups and subnilmanifolds]
Let $G/\Gamma$ be a nilmanifold with Mal'cev basis $\mathcal{X}=\{X_1,\ldots,X_m\}$. Suppose that $G'$ is a closed connected subgroup of $G$. We say that $G'$ is \emph{$Q$-rational} relative to $\mathcal{X}$ if the Lie algebra $\g'$ has a basis consisting of $Q$-rational combinations of the $X_i$. In this case, the subgroup $\Gamma'$, defined to be $G'\cap\Gamma$, is a discrete cocompact subgroup of $G'$, and so $G'/\Gamma'$ is a nilmanifold. We call $G'/\Gamma'$ a \emph{subnilmanifold} of $G/\Gamma$.
\end{definition}

We must also define a way in which a polynomial sequence can be `almost constant'.
\begin{definition}[Smooth sequences]
Let $M,N\ge1$. Let $G/\Gamma$ be a nilmanifold with Mal'cev basis $\mathcal X$, and let $d$ be the metric on $G$ associated to $\mathcal X$. Let $(\varepsilon(n))_{n\in\Z}$ be a sequence in $G$. Then we say that $(\varepsilon(n))_{n\in\Z}$ is \emph{$(M,N)$-smooth} if
\[
d(\varepsilon(n),1)\le M\qquad\qquad\text{and}\qquad\qquad d(\varepsilon(n),\varepsilon(n-1))\le M/N
\]
for all $n\in[N]$.
\end{definition}

We can now finally state precisely the factorisation theorem of Green and Tao.
\begin{theorem}[Green--Tao {\cite[Theorem 1.19]{poly.seq}}]\label{thm:poly.seq}
Let $M_0,N>0$ and let $A>0$. Let $G/\Gamma$ be a nilmanifold with a filtration $G_\bullet$ of degree $d$ and let $\mathcal{X}$ be an $M_0$-rational Mal'cev basis for $G/\Gamma$ adapted to $G_\bullet$. Let $g:\Z\to G$ be a polynomial sequence such that $g(0)=1$. Then there exist an integer $M$ with $M_0\le M\le M_0^{O_{A,G,G_\bullet}(1)}$; a rational subgroup $G'\subset G$; a Mal'cev basis $\mathcal{X}'$ for $G'/\Gamma'$ that is adapted to some filtration of $G'$ and in which each element is an $M$-rational combination of the elements of $\mathcal{X}$; and a decomposition $g=\varepsilon g'\gamma$ into polynomial sequences $\varepsilon,g',\gamma:\Z\to G$  satisfying the following conditions:
\begin{enumerate}
\renewcommand{\labelenumi}{(\roman{enumi})}
\item $\varepsilon:\Z\to G$ is $(M,N)$-smooth;
\item $g'$ takes values in $G'$ and the finite sequence $(g'(n)\Gamma')_{n\in[N]}$ is totally $1/M^A$-equidistributed in $G'/\Gamma'$ with respect to $\mathcal{X}'$;
\item $(\gamma(n)\Gamma')_{n\in\Z}$ is periodic with period at most $M$.
\item $\varepsilon(0)=g'(0)=\gamma(0)=1$
\end{enumerate}
\end{theorem}
\begin{proof}[Remarks on the proof]
The statement of \cite[Theorem 1.19]{poly.seq} is not quite the same as the statement of Theorem \ref{thm:poly.seq}, in that there is no assumption that $g(0)=1$ and, correspondingly, there is no conclusion that $\varepsilon(0)=g'(0)=\gamma(0)=1$. One sees that the former implies the latter on inspection of the proof of \cite[Theorem 1.19]{poly.seq}. In fact, \cite[Theorem 1.19]{poly.seq} is an instance of \cite[Theorem 10.2]{poly.seq}. This in turn is obtained by repeated application of \cite[Theorem 9.2]{poly.seq}, which establishes \cite[Theorem 10.2]{poly.seq} in a certain special case.

The first part of the proof of \cite[Theorem 9.2]{poly.seq} reduces to the case in which the polynomial sequence under consideration takes value $1$ at $0$. Once in that case, it is straightforward to verify that the sequences $\varepsilon,g',\gamma$ arising from the proof also take the value $1$ at $0$. Indeed, the sequences $\varepsilon$ and $\gamma$ satisfy
\[
\psi(\gamma(n))=\sum_{j>0}v_j\binom{n}{j}\qquad\qquad\text{and}\qquad\qquad\psi(\varepsilon(n))=\sum_{j>0}v'_j\binom{n}{j}.
\]
Here, the $v_j,v'_j$ are certain real numbers, the values of which are superfluous for the purposes of this discussion since when $n=0$ the binomial coefficients appearing in the sums all take the value zero, and so $\gamma(0)$ and $\varepsilon(0)$ must both equal the identity. The condition $g=\varepsilon g'\gamma$ then implies that $g'(0)$ is also the identity, as claimed.

One could therefore simply require by definition that a polynomial sequence takes value $1$ at $0$, without affecting the truth of \cite[Theorem 9.2]{poly.seq}. The deduction of \cite[Theorem 10.2]{poly.seq} would proceed in exactly the same way as in \cite{poly.seq}, but with the additional conclusion that all polynomial sequences arising as a result would take the value $1$ at $0$.

The reader may also note that \cite[Theorem 1.19]{poly.seq} does not say explicitly that the Mal'cev basis $\mathcal{X}'$ for $G'/\Gamma'$ is adapted to a filtration of $G'$. However, this apparent omission is simply because the nomenclature of that paper is not quite the same as in this paper, in that in \cite{poly.seq} Mal'cev bases are, by definition, always adapted to some filtration. In fact, one sees from the proof of \cite[Theorem 10.2]{poly.seq} that the filtration of $G'$ to which $\mathcal{X}'$ is adapted is given by $G_\bullet\cap G'$.
\end{proof}
\begin{remark}
A slightly more careful inspection of the proof of \cite[Theorem 10.2]{poly.seq} reveals that, even in the absence of any assumption on $g(0)$, one can conclude that $\varepsilon(0)$ lies in the fundamental domain of $G/\Gamma$ and that $\gamma(0)\in\Gamma$, with $\varepsilon(0)\gamma(0)=g(0)$. Thus, in particular, if $g(0)$ is the identity then so too are $\varepsilon(0)$, $g'(0)$ and $\gamma(0)$.
\end{remark}
The following lemma, the proof of which we defer until Appendix \ref{sec:poly.seq}, gives an idea of how we will use the factorisation theorem of Green and Tao to deduce recurrence results for polynomial sequences.
\begin{lemma}\label{lem:equidist.to.recur}
Let $M\ge2$. Let $G/\Gamma$ be an $m$-dimensional nilmanifold with an $M$-rational nested Mal'cev basis $\mathcal{X}$, and let $d$ be the metric associated to $\mathcal{X}$. Let $\rho\le1$ and $x\in G$, and suppose that $g:[N]\to G$ is $\eta$-equidistributed in $G/\Gamma$. Then a proportion of at least
\[
\frac{\rho^m}{M^{O(m)}}-\frac{3\eta}{\rho}
\]
of the points $(g(n)\Gamma)_{n\in[N]}$ lie in the ball $\{y\Gamma:d(y\Gamma,x\Gamma)\le\rho\}$.
\end{lemma}
Lemma \ref{lem:equidist.to.recur}, of course, shows that the component $g'$ of the polynomial sequence $g$ given by Theorem \ref{thm:poly.seq} is recurrent in a certain sense. In the context of proving Theorem \ref{lem:general.Bohr.set.size.bound.lower}, however, it is the sequence $g$ itself that we will need to be recurrent. The following lemma allows us to obtain recurrence of $g$ from recurrence of $g'$. Again, we defer the proof until Appendix \ref{sec:poly.seq}.
\begin{lemma}\label{lem:g'.to.g}
Let $\rho\le1$ and $\sigma$ be parameters. Let $G/\Gamma$ be a nilmanifold with an $M$-rational nested Mal'cev basis $\mathcal{X}$, suppose that $G'$ is a rational subgroup of $G$, and suppose that $\mathcal{X}'$ is a nested Mal'cev basis for $G'/\Gamma'$ in which each element is an $M$-rational combination of the elements of $\mathcal{X}$. Suppose that $\varepsilon\in G$ satisfies $d(\varepsilon,1)\le\sigma$, and that $\gamma\in\Gamma$. Finally, suppose that $g$ is an element of $G'$ such that $d'(g\Gamma',\Gamma')\le\rho$. Then $d(\varepsilon g\gamma\Gamma,\Gamma)\le M^{O(1)}\rho+\sigma$.
\end{lemma}
An immediate issue with combining Theorems \ref{thm:berg.leib} and \ref{thm:poly.seq} is that the Mal'cev basis $\mathcal{Y}$ given by Theorem \ref{thm:berg.leib} is not necessarily adapted to the filtration $G_\bullet$ given by Lemma \ref{lem:finer.filt}. However, the following result shows that the coordinate system associated to $\mathcal{Y}$ is at least comparable to the metric associated to the standard basis, which \emph{is} a Mal'cev basis adapted to $G_\bullet$.
\begin{lemma}\label{lem:std.basis}
Let $\mathcal{Y}=\{Y_1,\ldots,Y_m\}$ be a nested Mal'cev basis for $T_p^r/Z_p^r$ in which each element $Y_i$ is equal to either an element $X_j$ of the standard basis $\mathcal{X}$ or its inverse $-X_j$. Then the nilmanifold coordinate system $\chi_\mathcal{Y}$ associated to $\mathcal{Y}$, and the metric $d$ associated to the standard basis $\mathcal{X}=\{X_1,\ldots,X_m\}$, satisfy
\[
|\chi_\mathcal{Y}(x)|\ll_{p,r}d(x\Gamma,\Gamma)
\]
for every $x\in T_p^r$.
\end{lemma}
The proofs of Lemmas \ref{lem:equidist.to.recur}, \ref{lem:g'.to.g} and \ref{lem:std.basis} all essentially proceed by piecing together various results from \cite[Appendix A]{poly.seq}. We present the details in Appendix \ref{sec:poly.seq}. Modulo these proofs, it is now a fairly straightforward matter to combine the results of this section to prove Theorem \ref{lem:general.Bohr.set.size.bound.lower}, as follows. Recall that Proposition \ref{prop:recur.implies.non-uniform} and Theorem \ref{lem:general.Bohr.set.size.bound.lower} combine to give Theorem \ref{thm:bracket.non.uniform}.
\begin{proof}[Proof of Theorem \ref{lem:general.Bohr.set.size.bound.lower}]
Apply Theorem \ref{thm:berg.leib} and Lemma \ref{lem:std.basis} to obtain a constant-free polynomial form $P$ on a nilmanifold $G/\Gamma:=T_p^r/Z_p^r$ with a nested Mal'cev basis $\mathcal{Y}=\{Y_1,\ldots,Y_m\}$ such that
\begin{equation}\label{eq:Theta=xi}
\{\pm\Theta_i(n)\}=\chi_\mathcal{Y}(P(n))_{m-r+i},
\end{equation}
and such that the nilmanifold coordinate map $\chi_\mathcal{Y}$ associated to $\mathcal{Y}$ and the metric $d$ associated to the standard basis $\mathcal{X}=\{X_1,\ldots,X_m\}$ satisfy $|\chi_\mathcal{Y}(x)|\ll_{p,r}d(x\Gamma,\Gamma)$ for every $x\in G$. In fact, since $p$ and $r$ depend only on $\Theta_1,\ldots,\Theta_r$ we have
\begin{equation}\label{eq:std.basis}
|\chi_\mathcal{Y}(x)|\ll_{{\Theta_1},\ldots,{\Theta_r}}d(x\Gamma,\Gamma).
\end{equation}
Applying Lemma \ref{lem:finer.filt}, let $G_\bullet$ be a filtration of $G$ to which the standard basis is adapted, and with respect to which every realisation of $P$ is a polynomial sequence.

Let $A>0$ be a constant to be chosen later but depending only on $\Theta_1,\ldots,\Theta_r$ and $\delta$. Fix arbitrary realisations $\theta_1,\ldots,\theta_r$ of $\Theta_1,\ldots,\Theta_r$, respectively. Let $g$ be a realisation of $P$ such that
\begin{equation}\label{eq:theta=xi}
\{\pm\theta_i(n)\}=\chi_\mathcal{Y}(g(n))_{m-r+i}
\end{equation}
for $i=1,\ldots,r$; such a realisation exists by (\ref{eq:Theta=xi}). Note in particular that the fact that $P$ is constant free implies that $g(0)=1$.
Applying Theorem \ref{thm:poly.seq} with $M_0=2$ therefore gives an integer $M$ with $2\le M\ll_{A,G,G_\bullet}1$; a rational subgroup $G'\subset G$; a Mal'cev basis $\mathcal{X}'$ for $G'/\Gamma'$ adapted to some filtration of $G'$ and in which each element is an $M$-rational combination of the elements of $\mathcal{X}$; and a decomposition $g=\varepsilon g'\gamma$ into polynomial sequences $\varepsilon,g',\gamma:\Z\to G$ that satisfy conditions (i), (ii), (iii) and (iv) of Theorem \ref{thm:poly.seq}. In fact, since $A$ depends only on $\Theta_1,\ldots,\Theta_r$ and $\delta$, and since $G$ and $G_\bullet$ depend only on $\Theta_1,\ldots,\Theta_r$, we may assume that
\begin{equation}\label{eq:M}
2\le M\ll_{{\Theta_1},\ldots,{\Theta_r},\delta}1.
\end{equation}

Condition (iii) states that $(\gamma(n)\Gamma)_{n\in\Z}$ is periodic with period $q$, say, with $q\le M$. This last inequality, combined with the upper bound (\ref{eq:M}) on $M$, implies that in order to prove the proposition it is sufficient to show that a fraction $\Omega_{\Theta_1,\ldots,\Theta_r,\delta}(1)$ of the points in $q\N\cap[N]$ belong to the set $B_N(\theta_1,\ldots,\theta_r;I_\delta,\ldots,I_\delta)$, and so we may restrict attention to $q\N\cap[N]$ if we wish. Let us do so, replacing each $\theta_i$ by the function $\hat\theta_i$ defined by $\hat\theta_i(n)=\theta_i(qn)$; replacing $g$ by the sequence $\hat g$ defined by $\hat g(n)=g(qn)$; and replacing $N$ by $\lfloor N/q\rfloor$. Once these replacements are made, conditions (i), (ii), (iii) and (iv) of Theorem \ref{thm:poly.seq} become the following conditions:
\begin{enumerate}
\renewcommand{\labelenumi}{(\roman{enumi})}
\item $\varepsilon:\Z\to G$ is $(M,N)$-smooth;
\item $g'$ takes values in $G'$, and for any progression $Q\subset[N]$ of length at least $N/M^{A-1}$ the sequence $(g'(n)\Gamma')_{n\in Q}$ is $1/M^A$-equidistributed in $G'/\Gamma'$ with respect to $\mathcal{X}'$;
\item $\gamma$ takes values in $\Gamma$;
\item $\varepsilon(0)=g'(0)=\gamma(0)=1$.
\end{enumerate}
Since $A$ depends only on $\Theta_1,\ldots,\Theta_r$ and $\delta$, the inequality (\ref{eq:M}) implies that we may restrict attention to the subsequence $[N/M^{A-2}]$. We may therefore replace $N$ by $\lfloor N/M^{A-2}\rfloor$, and hence replace conditions (i) and (ii) by the following:
\begin{enumerate}
\renewcommand{\labelenumi}{(\roman{enumi})}
\item $d(\varepsilon(n),1)\le1/M^{A-3}$ for every $n\in[N]$;
\item $g'$ takes values in $G'$, and the finite sequence $(g'(n)\Gamma')_{n\in[N]}$ is $1/M^A$-equidistributed in $G'/\Gamma'$ with respect to $\mathcal{X}'$.
\end{enumerate}
Let $\rho\le1$ be a constant to be determined later. Lemma \ref{lem:equidist.to.recur} and condition (ii) together imply that
\begin{equation}\label{eq:prob.in.ball}
\Prob_{n\in[N]}(d'(g'(n)\Gamma',\Gamma')\le\rho)\ge\frac{\rho^m}{M^{O(m)}}-\frac{3}{\rho M^A}.
\end{equation}
Condition (i), on the other hand, combines with Lemma \ref{lem:g'.to.g} to imply that whenever $n\in[N]$ and $d'(g'(n)\Gamma',\Gamma')\le\rho$ we have $d(\varepsilon(n)g'(n)\gamma(n)\Gamma,\Gamma)\le M^{O(1)}\rho+1/M^{A-3}$. It therefore follows from (\ref{eq:prob.in.ball}) that
\[
\Prob_{n\in[N]}\left(d(g(n)\Gamma,\Gamma)\le M^{O(1)}\rho+1/M^{A-3}\right)\ge\frac{\rho^m}{M^{O(m)}}-\frac{3}{\rho M^A}.
\]
In light of (\ref{eq:std.basis}) and the lower bound (\ref{eq:M}) on $M$, this in turn implies that there exists a constant $C\ge1$ depending only on $\Theta_1,\ldots,\Theta_r$ such that
\[
\Prob_{n\in[N]}\left(|\chi_\mathcal{Y}(g(n))|\le M^C\left(\rho+\frac{1}{M^A}\right)\right)\ge\frac{\rho^m}{M^{Cm}}-\frac{3}{\rho M^A}.
\]
Setting $\rho=\delta/2M^C$ therefore implies that
\[
\Prob_{n\in[N]}\left(|\chi_\mathcal{Y}(g(n))|\le\frac{\delta}{2}+M^{C-A}\right)\ge\frac{\delta^m}{2^mM^{2Cm}}-\frac{6M^{C-A}}{\delta}.
\]
Thanks to the lower bound (\ref{eq:M}) on $M$, by setting $A$ sufficiently large in terms of $m$ and $C$, which depend only on $\Theta_1,\ldots,\Theta_r$, and $\delta$, we may therefore conclude that
\[
\Prob_{n\in[N]}(|\chi_\mathcal{Y}(g(n))|\le\delta)\ge\frac{\delta^m}{M^{O_{{\Theta_1},\ldots,{\Theta_r}}(m)}}.
\]
It follows from (\ref{eq:theta=xi}), the upper bound (\ref{eq:M}) on $M$ and the fact that $m$ depends only on $\Theta_1,\ldots,\Theta_r$ that
\[
\Prob_{n\in[N]}(\|\theta_i(n)\|_{\R/\Z}\le\delta\text{ for each }i=1,\ldots,r)\gg_{{\Theta_1},\ldots,{\Theta_r},\delta}1,
\]
and so the proposition is proved.
\end{proof}
\section{Weak recurrence of bracket polynomials}\label{sec:weak.recur}
If one includes the work of Bergelson--Leibman and Green--Tao that we used to prove Theorem \ref{lem:general.Bohr.set.size.bound.lower}, the proof of Theorem \ref{thm:bracket.non.uniform} is extremely long and difficult. In this section and Section \ref{sec:approx.loc.poly} we investigate the extent to which we can prove similar results using only elementary methods.

We concentrate our attention on an explicit model setting. Specifically, we consider the bracket polynomials $\phi_{k-1}$ defined by $\phi_{k-1}(n)=\alpha_{k-1}n\{\alpha_{k-2}n\{\ldots\{\alpha_1n\}\ldots\}\}$. Theorem \ref{thm:bracket.non.uniform} instantly tells us that
\begin{equation}\label{eq:anbncndn.U5}
\|e(\phi_{k-1})\|_{U^k[N]}\gg_k1.
\end{equation}
Our elementary methods will allow us to prove (\ref{eq:anbncndn.U5}) directly in the cases $k\le5$.

For $k\le3$ there is already nothing more to do. Indeed, the case $k=2$ is trivial, whilst the case $k=3$ follows from Proposition \ref{prop:recur.implies.non-uniform} and Lemma \ref{lem:linear.Bohr.set.size.bound.lower}. When $k\ge4$, however, $\phi_{k-1}$ has the non-linear bracket component $\{\alpha_2n\{\alpha_1n\}\}$, and so Lemma \ref{lem:linear.Bohr.set.size.bound.lower} does not apply to the set of bracket components of $\phi_{k-1}$. The cases $k=4,5$ therefore require some more work. We treat the case $k=4$ in this section, and then prove the case $k=5$ in Section \ref{sec:approx.loc.poly}.

Theorem \ref{lem:general.Bohr.set.size.bound.lower} is a very general result, but it is also somewhat stronger than is strictly necessary to prove Theorem \ref{thm:bracket.non.uniform}. For that purpose it would in fact be sufficient to establish recurrence of bracket polynomials in a weaker sense.
\begin{definition}[Weak recurrence (modulo 1)]
A set $\{\nu_1,\ldots,\nu_m\}$ of bracket polynomials on $[N]$ will be said to be \emph{$(\varepsilon,\lambda)$-weakly recurrent (modulo 1)} if
\[
|B_N(\nu_1,\ldots,\nu_m;I_{\frac{1}{2}-\varepsilon},\ldots,I_{\frac{1}{2}-\varepsilon})|\ge\lambda N.
\]
\end{definition}
Whilst the first notion of recurrence that we discussed required a positive fraction of the values of $\{\nu_i(n)\}$ to be very close to zero, weak recurrence requires only that a positive fraction of the values of $\{\nu_i(n)\}$ are not too close to $1/2$.

An elementary weak-recurrence result for arbitrary finite sets of bracket polynomials would give an elementary proof that bracket polynomials were non-uniform, thanks to the following result.
\begin{prop}\label{prop:weak.recur.implies.strong.poly}Let $\phi:[N]\to\R$ be a bracket polynomial of degree at most $k$. Suppose that $C(\phi)=\{\nu_1,\ldots,\nu_m\}$ is $(\varepsilon,\lambda)$-weakly recurrent (modulo 1). Then there exist intervals $J_1,\ldots,J_m\subset(-1/2,1/2]$ such that $|B_N(\nu_1,\ldots,\nu_m;J_1,\ldots,J_m)|\gg_{k,m,\varepsilon,\lambda}N$, and such that $\phi$ is strongly locally polynomial of degree at most $k$ on $B_N(\nu_1,\ldots,\nu_m;J_1,\ldots,J_m)$.
\end{prop}
\begin{proof}
By weak recurrence we have
\[
|B_N(\nu_1,\ldots,\nu_m;I_{\frac{1}{2}-\varepsilon},\ldots,I_{\frac{1}{2}-\varepsilon})|\ge\lambda N.
\]
Let $c_k$ be the constant given by Lemma~\ref{lem:strong.local.poly.set}, and set $\delta=\min\{c_k,\varepsilon k^{-1}\}$. Applying Lemma~\ref{lem:pigeonhole.mod.1} to $\nu_1,\ldots,\nu_m$ with $A=B_N(\nu_1,\ldots,\nu_m;I_{\frac{1}{2}-\varepsilon},\ldots,I_{\frac{1}{2}-\varepsilon})$, with $\delta_1=\ldots=\delta_m=\delta$ and $I=I_{\frac{1}{2}-\varepsilon}$ gives intervals $J_1,\ldots,J_m$ of width $\delta$ inside $I_{\frac{1}{2}-\varepsilon}$ such that
\[
|B_N(\nu_1,\ldots,\nu_m;J_1,\ldots,J_m)|\gg_{k,m,\varepsilon,\lambda}N,
\]
Lemma~\ref{lem:strong.local.poly.set} then implies the desired result.
\end{proof}
\begin{remark}\label{rem:weak.recur.implies.large.Uk}
Combining Proposition~\ref{prop:weak.recur.implies.strong.poly} with Proposition~\ref{prop:strong.local.poly} shows that if $f(n):=e(\phi(n))$ for some $\phi$ of degree at most $k$ with $(\varepsilon,\lambda)$-weakly recurrent $C(\phi)=\{\nu_1,\ldots,\nu_m\}$ then $\|f\|_{U^{k+1}[N]}\gg_{k,m,\varepsilon,\lambda}1$.
\end{remark}
Proving weak recurrence (modulo 1) for an arbitrary bracket polynomial without appealing to the work of Bergelson--Leibman and Green--Tao appears to be somewhat difficult. However, building on Lemma~\ref{lem:linear.Bohr.set.size.bound.lower}, we are at least able to make some progress in the case that the set of bracket polynomials under consideration has at most one non-linear member.
\begin{prop}[Bracket linears and their product are weakly recurrent]\label{prop:low.depth.weak.recur}
Let $k,m,r\in\Z$ with $0\le m\le k,r$ and let $\alpha_0,\ldots,\alpha_r\in\R$. For $i=1,\ldots,r$ let $\nu_i:n\mapsto\{\alpha_in\}$ be a linear bracket polynomial. Let $\phi(n):=\alpha_0n^{k-m}\prod_{i=1}^m\nu_i(n)$. Then for $\varepsilon>0$ and $\varepsilon'\ll_{k,r}1$ we have
\[
|B_N(\phi,\nu_1,\ldots,\nu_r;I_{\frac{1}{2}-\varepsilon'},I_\varepsilon,\ldots,I_\varepsilon)|\gg_{k,r,\varepsilon}N.
\]
\end{prop}
We in fact prove Proposition~\ref{prop:low.depth.weak.recur} in the following form, which easily implies Proposition~\ref{prop:low.depth.weak.recur} when combined with Lemma~\ref{lem:linear.Bohr.set.size.bound.lower}.
\begin{prop}\label{prop:low.depth.not.concentrated}
Let $k,m,r\in\Z$ with $0\le m\le k,r$ and let $\alpha_0,\ldots,\alpha_r\in\R$. For $i=1,\ldots,r$ let $\nu_i:n\mapsto\{\alpha_in\}$ be a linear bracket polynomial. Let $\phi(n):=\alpha_0n^{k-m}\prod_{i=1}^m\nu_i(n)$. Let $\varepsilon\ll_k1$ and $\delta$ be parameters. Then there exists a real number $\eta\in(0,1)$, depending only on $k$ and $\varepsilon$, such that if there is some interval $J\subset(-1/2,1/2]$ of width at most $\delta$ for which
\[
|B_N(\phi,\nu_1,\ldots,\nu_r;J,I_\varepsilon,\ldots,I_\varepsilon)|\ge(1-\eta)|B_N(\nu_1,\ldots,\nu_r;I_\varepsilon,\ldots,I_\varepsilon)|
\]
then
\[
|B_N(\phi,\nu_1,\ldots,\nu_r;I_{O_k(\delta)},I_{O_k(\varepsilon)},\ldots,I_{O_k(\varepsilon)})|\gg_{k,r,\varepsilon}N.
\]
\end{prop}
We make use of two lemmas. We will say that a bracket polynomial is of degree \emph{exactly} $k$ if it is of degree at most $k$ but not of degree at most $k-1$.
\begin{lemma}\label{lem:kth.deriv.of.degree.k}
Suppose $\phi$ is an elementary bracket polynomial of degree exactly $k$ with bracket components $\nu_1,\ldots,\nu_m$. Then there exists $\hat{c}_k$ such that if $\varepsilon\le\hat{c}_k$ and $n,n+h,\ldots,n+kh$ all lie in $B_N(\nu_1,\ldots,\nu_m;I_\varepsilon,\ldots,I_\varepsilon)$ we have
\[
(\Delta_h)^k\phi(n)=k!\phi(h).
\]
\end{lemma}
The proof is a simple induction and left as an exercise to the reader.
\begin{lemma}\label{lem:kAPs.of.degree.k}
Let $k,m,r\in\Z$ with $0\le m\le k,r$ and let $\alpha_0,\ldots,\alpha_r\in\R$. For $i=1,\ldots,r$ let $\nu_i:n\mapsto\{\alpha_in\}$ be a linear bracket polynomial. Let $\phi(n):=\alpha_0n^{k-m}\prod_{i=1}^m\nu_i(n)$. Let $\hat{c}_k$ be as in Lemma \ref{lem:kth.deriv.of.degree.k}, and let $\varepsilon\le\hat{c}_k$ and $\delta>0$ be parameters. Suppose that $n\in[N]$ and $h>0$ and that there exists an interval $J\subset(-1/2,1/2]$ of width at most $\delta$ such that $n,n+h,\ldots,n+kh$ all lie in $B_{\lfloor N/k!\rfloor}(\phi,\nu_1,\ldots,\nu_r;J,I_\varepsilon,\ldots,I_\varepsilon)$. Then
\[
k!h\in B_N(\phi,\nu_1,\ldots,\nu_r;I_{O_k(\delta)},I_{O_k(\varepsilon)},\ldots,I_{O_k(\varepsilon)}).
\]
\end{lemma}
\begin{proof}
The fact that $n,n+h,\ldots,n+kh\in B_{\lfloor N/k!\rfloor}(\phi,\nu_1,\ldots,\nu_r;J,I_\varepsilon,\ldots,I_\varepsilon)$ implies in particular that $n,n+h,\ldots,n+kh\in B_N(\nu_1,\ldots,\nu_m;I_\varepsilon,\ldots,I_\varepsilon)$, and so Lemma~\ref{lem:kth.deriv.of.degree.k} implies that
\begin{equation}\label{eq:kAPs.of.degree.k.1}
(\Delta_h)^k\phi(n)=k!\phi(h).
\end{equation}
However, the fact that ${\phi(n+jh)}\in J$ for all $j$ implies, by Lemma~\ref{lem:frac.part.difference} (i) and induction, that
\[
(\Delta_h)^k\phi(n)\in I_{O_k(\delta)},
\]
which combined with (\ref{eq:kAPs.of.degree.k.1}) of course implies that
\begin{equation}\label{eq:kAPs.of.degree.k.2}
k!\phi(h)\in I_{O_k(\delta)}.
\end{equation}
Now the fact that $\nu_i(n)$ and $\nu_i(n+h)$ both lie in $I_\varepsilon$ for all $i$ implies, by Lemma~\ref{lem:frac.part.difference} (i), that
\begin{equation}\label{eq:kAPs.of.degree.k.3}
\nu_i(h)\in I_{2\varepsilon}\text{ for all $i$},
\end{equation}
which in turn implies, by Lemma~\ref{lem:frac.part.difference} (ii), that $\nu_i(k!h)=k!\nu_i(h)$ for all $i$. This implies that $\phi(k!h)=(k!)^k\phi(h)$, which combined with (\ref{eq:kAPs.of.degree.k.2}) and Lemma~\ref{lem:frac.part.difference} (ii) gives
\begin{equation}\label{eq:kAPs.of.degree.k.4}
\phi(k!h)\in I_{O_k(\delta)}.
\end{equation}
Furthermore, (\ref{eq:kAPs.of.degree.k.3}) and Lemma~\ref{lem:frac.part.difference} (ii) imply that
\begin{equation}\label{eq:kAPs.of.degree.k.5}
\nu_i(k!h)\in I_{O_k(\varepsilon)}\text{ for all $i$}.
\end{equation}
Combining (\ref{eq:kAPs.of.degree.k.4}) and (\ref{eq:kAPs.of.degree.k.5}) yields the desired result.
\end{proof}
\begin{proof}[Proof of Proposition~\ref{prop:low.depth.not.concentrated}]
\begin{sloppypar}
By Lemma~\ref{lem:kAPs.of.degree.k} it suffices to find $\Omega_{k,r,\varepsilon}(N)$ values of $h$ for each of which there exists at least one progression $n,n+h,\ldots,n+kh$ of common difference $h$ contained within $B_{\lfloor N/k!\rfloor}(\phi,\nu_1,\ldots,\nu_r;J,I_\varepsilon,\ldots,I_\varepsilon)$. We can certainly find many values of $h$ for which there exist such progressions contained within $B_{\lfloor N/k!\rfloor}(\nu_1,\ldots,\nu_r;I_\varepsilon,\ldots,I_\varepsilon)$; indeed, by Lemma~\ref{lem:linear.Bohr.set.size.bound.lower} we have
\begin{equation}\label{eq:low.depth.not.concentrated.1}
|B_{\lfloor N/(k+1)!\rfloor}(\nu_1,\ldots,\nu_r;I_{\varepsilon/(k+1)},\ldots,I_{\varepsilon/(k+1)})|\gg_{k,r,\varepsilon}N,
\end{equation}
and if $h\in B_{\lfloor N/(k+1)!\rfloor}(\nu_1,\ldots,\nu_r;I_{\varepsilon/(k+1)},\ldots,I_{\varepsilon/(k+1)})$ then for \emph{every} $n\in B_{\lfloor N/(k+1)!\rfloor}(\nu_1,\ldots,\nu_r;I_{\varepsilon/(k+1)},\ldots,I_{\varepsilon/(k+1)})$ we have $n,n+h,\ldots,n+kh\in B_{\lfloor N/k!\rfloor}(\nu_1,\ldots,\nu_r;I_\varepsilon,\ldots,I_\varepsilon)$. Writing $a=a_{k,r,\varepsilon}$ for the constant implicit in (\ref{eq:low.depth.not.concentrated.1}), so that
\[
|B_{\lfloor N/(k+1)!\rfloor}(\nu_1,\ldots,\nu_r;I_{\varepsilon/(k+1)},\ldots,I_{\varepsilon/(k+1)})|>aN,
\]
we may conclude that for at least $\Omega_{k,r,\varepsilon}(N)$ values of $h$ we have at least $aN$ values of $n$ for which $n,n+h,\ldots,n+kh\in B_{\lfloor N/k!\rfloor}(\nu_1,\ldots,\nu_r;I_\varepsilon,\ldots,I_\varepsilon)$.

Fix such an $h$. Now $B_{\lfloor N/k!\rfloor}(\phi,\nu_1,\ldots,\nu_r;J,I_\varepsilon,\ldots,I_\varepsilon)$ contains all but $\eta N$ of the points in $B_{\lfloor N/k!\rfloor}(\nu_1,\ldots,\nu_r;I_\varepsilon,\ldots,I_\varepsilon)$, and each element of $B_{\lfloor N/k!\rfloor}(\nu_1,\ldots,\nu_r;I_\varepsilon,\ldots,I_\varepsilon)$ can belong to at most $k+1$ progressions $n,n+h,\ldots,n+kh$, and so $B_{\lfloor N/k!\rfloor}(\phi,\nu_1,\ldots,\nu_r;J,I_\varepsilon,\ldots,I_\varepsilon)$ contains at least all but $(k+1)\eta N$ of the progressions $n,n+h,\ldots,n+kh$ in $B_{\lfloor N/k!\rfloor}(\nu_1,\ldots,\nu_r;I_\varepsilon,\ldots,I_\varepsilon)$.

In particular, if we fix $\eta<a/(k+1)$ then the set $B_{\lfloor N/k!\rfloor}(\phi,\nu_1,\ldots,\nu_r;J,I_\varepsilon,\ldots,I_\varepsilon)$ contains at least one such progression. The value of $h$ was chosen arbitrarily from a set of cardinality $\Omega_{k,r,\varepsilon}(N)$, and so the proposition is proved.
\end{sloppypar}
\end{proof}
\begin{proof}[Proof of (\ref{eq:anbncndn.U5}) in the case $k=4$]The function $\phi_2$ is of the form required in order to apply Proposition~\ref{prop:low.depth.weak.recur}, and so the bracket components of the function $\phi_3$ are weakly recurrent (modulo 1). The case $k=4$ then follows from Remark~\ref{rem:weak.recur.implies.large.Uk}.
\end{proof}
\begin{remark}
More generally, and by an identical proof, the function
\[
f:n\mapsto e\left(\gamma n^t\left\{\beta n^r\prod_{i=1}^m\{\alpha_in\}\right\}^s\right)
\]
satisfies $\|f\|_{U^{s(r+m)+t+1}[N]}\gg_{m,r,s,t}1$.
\end{remark}
\section{Approximately locally polynomial functions}\label{sec:approx.loc.poly}
We can push slightly further than Section~\ref{sec:weak.recur} and prove the $k=5$ case of (\ref{eq:anbncndn.U5}) by relaxing the definition of being locally polynomial of degree $k-1$. The most obvious modification is to require the $k$th derivatives to vanish only modulo 1, since it is only their value modulo 1 that will affect the quantity $e(\Delta_{h_1,\ldots,h_k}\phi(n))$. Indeed, as was remarked in the introduction, we have been considering bracket polynomials as functions into $\R$, rather than into $\R/\Z$, only because it made some of the proofs cleaner in earlier sections.

\begin{sloppypar}Another natural way in which it is possible to weaken the definition is not even to require the derivatives to vanish (modulo 1), but instead to require that $\|\Delta_{h_1,\ldots,h_k}\phi(n)\|_{\R/\Z}<\delta$ for some $\delta\ll1$, as this would still be sufficient to introduce some bias into the sum $\E_{n\in[N],h\in[-N,N]^k}e(\Delta_{h_1,\ldots,h_k}\phi(n))$.\end{sloppypar}
\begin{definition}[Approximately locally polynomial (modulo 1)]
Let $\phi:[N]\to\R$ be a function and let $B\subset[N]$. Then $\phi$ is said to be \emph{$\delta$-approximately locally polynomial of degree $k-1$ (modulo 1)} on $B$ if whenever $n\in[N]$ and $h\in[-N,N]^k$ satisfy $n+\omega\cdot h\in B$ for all $\omega\in\{0,1\}^k$ we have
\begin{equation}\label{eq:approx.loc.poly}
\|\Delta_{h_1,\ldots,h_k}\phi(n)\|_{\R/\Z}\le\delta.
\end{equation}
If (\ref{eq:approx.loc.poly}) holds whenever $n+\omega\cdot h\in B$ for all $\omega\in\{0,1\}^k\backslash\{\mathbf0\}$ then $\phi$ is said to be \emph{strongly $\delta$-approximately locally polynomial of degree at most $k-1$ (modulo 1)} on $B$.
\end{definition}
The utility of making this definition lies in the following result.
\begin{prop}\label{prop:approx.strong.local.poly}
Let $\phi:[N]\to\R$ be a function and define $f:[N]\to\C$ by $f(n):=e(\phi(n))$. Let $\delta\in[0,1/4)$ be a parameter and suppose that $\phi$ is strongly $\delta$-approximately locally polynomial of degree at most $k$ (modulo 1) on some set $B\subset[N]$ with $|B|\gg N$. Then $\|f\|_{U^{k+1}[N]}\gg_{k,\delta}1$.
\end{prop}
\begin{proof}
The definition of being strongly $\delta$-approximately locally polynomial of degree $k$ (modulo 1) on $B$ implies that $\operatorname{Re}(e(\Delta_{h_1,\ldots,h_{k+1}}\phi(n)))\gg_\delta1$ whenever $n+\omega\cdot h\in B$ for each $\omega\in\{0,1\}^{k+1}\backslash\{\mathbf0\}$, and so
\begin{displaymath}
\begin{split}
\left|\E_{n\in[N],h\in[-N,N]^{k+1}}\left(e(\Delta_{h_1,\ldots,h_{k+1}}\phi(n))\textstyle\prod_{\omega\in\{0,1\}^{k+1}\backslash\{\mathbf0\}}1_B(n+\omega\cdot h)\right)\right|\qquad\qquad\qquad\qquad\\
\gg_\delta\Prob_{n\in[N],h\in[-N,N]^{k+1}}(n+\omega\cdot h\in B\text{ for all }\omega\in\{0,1\}^{k+1}\backslash\{\mathbf0\}).
\end{split}
\end{displaymath}
However, this last quantity is trivially at least $\Prob_{n\in[N],h\in[-N,N]^{k+1}}(n+\omega\cdot h\in B\text{ for all }\omega\in\{0,1\}^{k+1})$, which by Lemma~\ref{lem:boxes.for.gcs} is at least $\Omega_k(1)$. An application of Lemma \ref{lem:application.of.gcs} therefore completes the proof.
\end{proof}
\begin{lemma}\label{lem:simple.deriv}Suppose that $\nu:[N]\to\R$ is a bracket polynomial that is strongly locally polynomial of degree $k-1$ on $A\subset[N]$, and define $\phi(n):=\lambda n\{\nu(n)\}$. Let $J\subset(-1/2,1/2]$ be an interval with $|J|\le2^{-k}$. Then whenever $n+\omega\cdot h\in B_N(\nu;J)\cap A$ for all $\omega\in\{0,1\}^{k+1}\backslash\{\mathbf0\}$ we have $\Delta_{h_1,\ldots,h_{k+1}}\phi(n)\in\{q\lambda n:q\in\Z,|q|\ll_k1\}$.
\end{lemma}
\begin{corollary}\label{cor:simple.deriv}
Suppose that $\nu:[N]\to\R$ is a bracket polynomial that is strongly locally polynomial of degree $k-1$ on $A\subset[N]$. Let $J\subset(-1/2,1/2]$ be an interval with $|J|\le2^{-k}$, and let $\delta>0$ be a parameter. Then the bracket polynomial $\phi:[N]\to\R$ defined by $\phi(n):=\lambda n\{\nu(n)\}$ is strongly $O_k(\delta)$-approximately locally polynomial of degree $k$ on $B_N(\lambda n,\nu;I_\delta,J)\cap A$.
\end{corollary}
\begin{proof}[Proof of Lemma~\ref{lem:simple.deriv}]
Assume that
\begin{equation}\label{eq:simple.deriv.0}
n+\omega\cdot h\in B_N(\nu;J)\cap A\text{ for all }\omega\in\{0,1\}^{k+1}\backslash\{\mathbf0\}.
\end{equation}
We have
\begin{equation}\label{eq:simple.deriv.1}
\Delta_{h_1,\ldots,h_{k+1}}\phi(n)=\sum_{\omega\in\{0,1\}^{k+1}}(-1)^{k+1-|\omega|}\lambda(n+\omega\cdot h)\{\nu(n+\omega\cdot h)\}.
\end{equation}
Splitting the right-hand side of (\ref{eq:simple.deriv.1}), we see that $\Delta_{h_1,\ldots,h_{k+1}}\phi(n)$ is equal to
\begin{equation}\label{eq:simple.deriv.2}
\begin{split}\lambda n\sum_{\omega\in\{0,1\}^{k+1}}(-1)^{k+1-|\omega|}\{\nu(n+\omega\cdot h)\}\qquad\qquad\qquad\qquad\\
+\sum_{i=1}^{k+1}\lambda h_i\sum_{\omega:\omega_i=1}(-1)^{k+1-|\omega|}\{\nu(n+\omega\cdot h)\}.
\end{split}
\end{equation}
However, the final sum of (\ref{eq:simple.deriv.2}) is equal to $\Delta_{h_1,\ldots,h_{i-1},h_{i+1},\ldots,h_{k+1}}\{\nu\}(n+h_i)$, which vanishes because $n+\omega\cdot h\in B_N(\nu;J)\cap A$ for every $\omega$ with $\omega_i=1$ and because $\{\nu\}$ is locally polynomial of degree $k-1$ on $B_N(\nu;J)\cap A$ by Lemma~\ref{lem:c.psi(n)} and the hypothesis that $|J|\le2^{-k}$. We therefore have
\begin{equation}\label{eq:simple.deriv.3}
\Delta_{h_1,\ldots,h_{k+1}}\phi(n)=\lambda n\sum_{\omega\in\{0,1\}^{k+1}}(-1)^{k+1-|\omega|}\{\nu(n+\omega\cdot h\})=\lambda n\Delta_{h_1,\ldots,h_{k+1}}\{\nu\}(n).
\end{equation}
Now $n$ may not belong to $B_N(\nu;J)\cap A$, and so we cannot similarly conclude that $\Delta_{h_1,\ldots,h_{k+1}}\{\nu\}(n)=0$. However, by (\ref{eq:simple.deriv.0}) and the assumption that $\nu$ is strongly locally polynomial on $A$ we can conclude that $\Delta_{h_1,\ldots,h_{k+1}}\nu(n)=0$, and it is clear that $\Delta_{h_1,\ldots,h_{k+1}}\{\nu\}(n)$ and $\Delta_{h_1,\ldots,h_{k+1}}\nu(n)$ differ by an integer and that $|\Delta_{h_1,\ldots,h_{k+1}}\{\nu\}(n)|\ll_k1$. Hence $\Delta_{h_1,\ldots,h_r}\{\nu\}(n)\in\{q\in\Z:|q|\ll_k1\}$, which combined with (\ref{eq:simple.deriv.3}) yields the desired result.
\end{proof}
\begin{proof}[Proof of (\ref{eq:anbncndn.U5}) in the case $k=5$]
Recall that
\[
\phi_{k-1}(n)=\alpha_{k-1}n\{\alpha_{k-2}n\{\ldots\{\alpha_1n\}\ldots\}\}.
\]
By Proposition~\ref{prop:low.depth.weak.recur} there exist $\varepsilon\ll1$ and $\varepsilon'\gg1$ such that
\[
|B_N(\phi_2,\phi_1,\alpha_4 n;I_{1/2-\varepsilon'},I_\varepsilon,I_\varepsilon)|\gg N,
\]
and so a similar argument to Proposition~\ref{prop:weak.recur.implies.strong.poly} implies that there is some interval $J\subset(-1/2,1/2]$ such that
\[
|B_N(\phi_2,\phi_1,\alpha_4 n;J,I_\varepsilon,I_\varepsilon)|\gg N
\]
and such that $\phi_3$ is strongly locally polynomial of degree 3 on $B_N(\phi_2,\phi_1,\alpha_4 n;J,I_\varepsilon,I_\varepsilon)$.

By Corollary~\ref{cor:simple.deriv} there exists $\delta\gg1$ such that if $J'$ is an interval in $(-1/2,1/2]$ of width $\delta$ then $\phi_4$ is, say, $1/10$-approximately strongly locally polynomial of degree 4 (modulo 1) on $B_N(\phi_3,\phi_2,\phi_1,\alpha_4 n;J',J,I_\varepsilon,I_\varepsilon)$. Applying the pigeonhole principle to the elements of $B_N(\phi_2,\phi_1,\alpha_4 n;J,I_\varepsilon,I_\varepsilon)$ we can obtain such an interval whilst ensuring that
\[
|B_N(\phi_3,\phi_2,\phi_1,\alpha_4 n;J',J,I_\varepsilon,I_\varepsilon)|\gg N.
\]
Proposition~\ref{prop:approx.strong.local.poly} then completes the proof of the theorem.
\end{proof}
\begin{remark}An identical proof shows, more generally, that the function
\[
f:n\mapsto e\left(\lambda n\left\{\gamma n^t\left\{\beta n^r\prod_{i=1}^m\{\alpha_in\}\right\}^s\right\}\right)
\]
satisfies $\|f\|_{U^{s(r+m)+t+2}[N]}\gg_{m,r,s,t}1$.
\end{remark}
\appendix
\section{Coordinates, metrics and equidistribution in nilmanifolds}
\label{sec:poly.seq}
The aim of this appendix is to prove Lemmas \ref{lem:equidist.to.recur}, \ref{lem:g'.to.g} and \ref{lem:std.basis}. Throughout, where $\mathcal{X}$ and $\mathcal{X}'$ are Mal'cev bases for a nilmanifold $G/\Gamma$ we write $d$ for the metrics on $G$ and $G/\Gamma$, and $\psi$ for the coordinates, associated to $\mathcal{X}$; we write $d'$ and $\psi'$, respectively, for the metrics and coordinates associated to $\mathcal{X}'$.

As we remarked in Section \ref{sec:literature}, the lemmas we are about to prove essentially follow by combining various results from \cite[Appendix A]{poly.seq}. The notation of that work is identical to ours, and so the results we cite can be read directly from \cite[Appendix A]{poly.seq} without difficulty. We therefore refer to these results by number only, without restating them here.

We repeatedly use the observation, made in the proof of \cite[Lemma A.15]{poly.seq}, that if $G/\Gamma$ is a nilmanifold then for every $x,y\in G$ there is some $z\in\Gamma$ such that $d(x\Gamma,y\Gamma)=d(x,yz)$.

We begin by recalling and proving Lemma \ref{lem:equidist.to.recur}.
\newtheorem*{lem:equidist.to.recur}{Lemma \ref{lem:equidist.to.recur}}
\begin{lem:equidist.to.recur}Let $M\ge2$. Let $G/\Gamma$ be an $m$-dimensional nilmanifold with an $M$-rational nested Mal'cev basis $\mathcal{X}$, and let $d$ be the metric associated to $\mathcal{X}$. Let $\rho\le1$ and $x\in G$, and suppose that $g:[N]\to G$ is $\eta$-equidistributed in $G/\Gamma$. Then a proportion of at least
\[
\frac{\rho^m}{M^{O(m)}}-\frac{3\eta}{\rho}
\]
of the points $(g(n)\Gamma)_{n\in[N]}$ lie in the ball $\{y\Gamma:d(y\Gamma,x\Gamma)\le\rho\}$.
\end{lem:equidist.to.recur}
We start by bounding from below the measure of a metric ball in $G/\Gamma$. Here and throughout this appendix we write $B_\rho(x)$ for the ball $\{y\Gamma:d(y\Gamma,x\Gamma)\le\rho\}$.
\begin{lemma}\label{lem:metric.ball}
Suppose $x\in G$ and let $\rho\in(0,1)$ be a parameter. Then
\[
\mu(B_\rho(x))\ge\frac{\rho^m}{M^{O(m)}}.
\]
\end{lemma}
\begin{proof}In this proof we appeal \cite[Lemma A.14]{poly.seq}. The reader may note that the hypothesis of that lemma includes the assumption that $\mathcal{X}$ is adapted to some filtration of $G$. However, the only place this is used is in invoking \cite[Lemma A.3]{poly.seq}, which assumes only the weaker property of being nested. We are therefore free to apply \cite[Lemma A.14]{poly.seq} in the context of Lemma \ref{lem:equidist.to.recur}.

Let $\delta\in(0,1)$ be a parameter to be determined later. Set $B'=\{y\in G:|\psi(y)-\psi(x)|\le\delta\}$. By \cite[Lemma A.14]{poly.seq} we may assume that $|\psi(x)|\le1$, and so \cite[Lemma A.4]{poly.seq} implies that there is an absolute constant $C$ such that for every $y\in B'$ we have
\[
d(y,x)\le M^C|\psi(y)-\psi(x)|\le M^C\delta.
\]
Setting $\delta=\rho/M^C$ therefore implies that $B'\Gamma\subset B_\rho(x)$, and in particular that $\mu(B)\ge\mu(B'\Gamma)$. It is a straightforward exercise to verify that for sufficiently small $\delta$ we have
\[
\mu(B'\Gamma)=\mu(B')=(2\delta)^m\ge\frac{\rho^m}{M^{Cm}},
\]
and so the lemma is proved.
\end{proof}
\begin{proof}[Proof of Lemma \ref{lem:equidist.to.recur}]
Define a non-negative function $f:G/\Gamma\to\R$ by
\[
f(y\Gamma)=\max\left\{0,1-\left(\textstyle\frac{2}{\rho}\right)d(y\Gamma,B_{\rho/2}(x))\right\}.
\]
Since $f$ takes the value $1$ on $B_{\rho/2}(x)$ we have
\[
\int_{G/\Gamma}f\ge\mu(B_{\rho/2}(x))\ge\frac{\rho^m}{M^{O(m)}}
\]
by Lemma \ref{lem:metric.ball}. Observe also that $f$ is Lipschitz with Lipschitz norm $1+2/\rho$, and so the $\eta$-equidistribution of $g$ therefore implies that
\[
\E_{n\in[N]}f(g(n)\Gamma)\ge\frac{\rho^m}{M^{O(m)}}-\eta\left(1+\frac{2}{\rho}\right)\ge\frac{\rho^m}{M^{O(m)}}-\frac{3\eta}{\rho}.
\]
The fact that $f$ is bounded by $1$ and supported on $B_\rho(x)$ therefore yields the desired result.
\end{proof}
We now recall and prove Lemma \ref{lem:g'.to.g}.
\newtheorem*{lem:g'.to.g}{Lemma \ref{lem:g'.to.g}}
\begin{lem:g'.to.g}
Let $\rho\le1$ and $\sigma$ be parameters. Let $G/\Gamma$ be a nilmanifold with an $M$-rational nested Mal'cev basis $\mathcal{X}$, suppose that $G'$ is a rational subgroup of $G$, and suppose that $\mathcal{X}'$ is a nested Mal'cev basis for $G'/\Gamma'$ in which each element is an $M$-rational combination of the elements of $\mathcal{X}$. Suppose that $\varepsilon\in G$ satisfies $d(\varepsilon,1)\le\sigma$, and that $\gamma\in\Gamma$. Finally, suppose that $g$ is an element of $G'$ such that $d'(g\Gamma',\Gamma')\le\rho$. Then $d(\varepsilon g\gamma\Gamma,\Gamma)\le M^{O(1)}\rho+\sigma$.
\end{lem:g'.to.g}
\begin{proof}
The fact that $d'(g\Gamma',\Gamma')\le\rho$ implies that there exists $z\in\Gamma'$ such that
\begin{equation}\label{eq:ap.1}
d'(gz,1)\le\rho.
\end{equation}
An application of \cite[Lemma A.4]{poly.seq} therefore implies that $|\psi'(gz)|\le M^{O(1)}\rho$, and so \cite[Lemma A.6]{poly.seq} and (\ref{eq:ap.1}) combine to give
\begin{equation}\label{eq:ap.2}
d(gz,1)\le M^{O(1)}\rho.
\end{equation}
The right-invariance of $d$ (\ref{eq:right.invar}) implies that $d(\varepsilon gz,1)=d(\varepsilon,(gz)^{-1})$, and so the symmetry of $d$ about the identity (\ref{eq:d.sym}) and the triangle inequality imply that
\begin{equation}\label{eq:ap.3}
d(\varepsilon gz,1)\le d(\varepsilon,1)+d((gz)^{-1},1)=d(\varepsilon,1)+d(gz,1).
\end{equation}
The left-hand side of (\ref{eq:ap.3}) is equal to $d(\varepsilon g\gamma\Gamma,\Gamma)$, since $\gamma,z\in\Gamma$, whilst the right-hand side is at most $M^{O(1)}\rho+\sigma$ by (\ref{eq:ap.2}) and the assumption on $\varepsilon$, and so the lemma is proved.
\end{proof}
Finally, let us recall and prove Lemma \ref{lem:std.basis}.
\newtheorem*{lem:std.basis}{Lemma \ref{lem:std.basis}}
\begin{lem:std.basis}
Let $\mathcal{Y}=\{Y_1,\ldots,Y_m\}$ be a nested Mal'cev basis for $T_p^r/Z_p^r$ in which each element $Y_i$ is equal to either an element $X_j$ of the standard basis $\mathcal{X}$ or its inverse $-X_j$. Then the nilmanifold coordinate map $\chi_\mathcal{Y}$ associated to $\mathcal{Y}$, and the metric $d$ associated to the standard basis $\mathcal{X}=\{X_1,\ldots,X_m\}$, satisfy
\[
|\chi_\mathcal{Y}(x)|\ll_{p,r}d(x\Gamma,\Gamma)
\]
for every $x\in T_p^r$.
\end{lem:std.basis}
\begin{proof}
Let $c<1$ be an absolute constant to be determined later. Since $|\chi_\mathcal{Y}(x)|\ll_{p,r}1$ for every $x\in G$, it is sufficient to prove the lemma under the additional assumption that
\begin{equation}\label{eq:ap.4}
d(x\Gamma,\Gamma)<c<1.
\end{equation}
Let $z$ be an element of $\Gamma$ satisfying
\begin{equation}\label{eq:ap.5}
d(xz,1)=d(x\Gamma,\Gamma).
\end{equation}
The assumptions on $\mathcal{Y}$ imply in particular that each element of $\mathcal{Y}$ is a $1$-rational combination of elements of the standard basis, and \emph{vice versa}, and so \cite[Lemma A.4]{poly.seq} combines with (\ref{eq:ap.4}) and (\ref{eq:ap.5}) to imply that there is an absolute constant $C$ such that 
\begin{equation}\label{eq:ap.6}
|\psi_\mathcal{Y}(xz)|\le Cd(x\Gamma,\Gamma).
\end{equation}
Setting $c=1/2C$, condition (\ref{eq:ap.4}) therefore implies that $|\psi_\mathcal{Y}(xz)|<1/2$, which in particular implies that
\[
\chi_\mathcal{Y}(x)=\psi_\mathcal{Y}(xz),
\]
and so the lemma follows from (\ref{eq:ap.6}).
\end{proof}
\section{Bergelson and Leibman's characterisation of bracket polynomials}
\label{sec:berg.leib}
The purpose of this appendix is to sketch how Theorem \ref{thm:berg.leib} can be read out of the work of Bergelson and Leibman \cite{berg.leib}. Let us begin, then, by recalling the statement of Theorem \ref{thm:berg.leib}.
\newtheorem*{thm:berg.leib}{Theorem \ref{thm:berg.leib}}
\begin{thm:berg.leib}[Bergelson--Leibman \cite{berg.leib}]
Let $\Theta_1,\ldots,\Theta_r$ be constant-free bracket forms. Then there exist $p\ge1$, a constant-free polynomial form $P$ on $T_p^r$, and a nested Mal'cev basis $\mathcal{Y}=\{Y_1,\ldots,Y_m\}$ for $T_p^r/Z_p^r$ such that each element $Y_i$ of $\mathcal{Y}$ is equal to either an element $X_j$ of the standard basis $\mathcal{X}$ or its inverse $-X_j$, and such that for every $i=1,\ldots,r$ we have $\{\pm\Theta_i\}=\chi_\mathcal{Y}(P)_{m-r+i}$.
\end{thm:berg.leib}
This essentially follows from \cite[Proposition 6.9]{berg.leib}. Indeed, it is shown in \cite[\S6.8]{berg.leib} how, given a Mal'cev basis $\mathcal{Y}$ of $T_p$, the nilmanifold coordinates $\chi_{\mathcal Y}(g)_i$ of a matrix $g\in T_p$ can be defined equivalently as formal bracket expressions in the entries of $g$; \cite[Proposition 6.9]{berg.leib} then states that if $\mathcal{A}$ is a commutative ring, and $b$ is an arbitrary bracket expression in the elements of $\mathcal{A}$, then there is some $T_p/Z_p$ with Mal'cev basis $\mathcal{Y}=\{Y_1,\ldots,Y_m\}$, and some upper-triangular matrix with elements of $\mathcal{A}$ as entries, such that $\{\pm b\}=\chi(A)_m$. To prove Theorem \ref{thm:berg.leib}, therefore, we essentially just apply this result with $\mathcal{A}$ as the ring of constant-free polynomial forms.

There are, however, some issues with this deduction.
\begin{enumerate}
\item In \cite{berg.leib} fractional parts are taken to lie in $[0,1)$, whereas in the present work they lie in $(-1/2,1/2]$.
\item Whilst it is explicit in \cite[Proposition 6.9]{berg.leib} each element $Y_i$ of the Mal'cev basis $\mathcal{Y}$ is equal to either an element $X_j$ of the standard basis $\mathcal{X}$ or its inverse $-X_j$, it is not stated explicitly that the $Y_i$ are ordered in such a way that $\mathcal Y$ is nested.
\item Applying \cite[Proposition 6.9]{berg.leib} gives only a single bracket polynomial in terms of a polynomial mapping into $T_p/Z_p$, rather than an $r$-tuple of bracket polynomials in terms of a polynomial mapping into $T_p^r/Z_p^r$.
\end{enumerate}
It is straightforward to check that the change in the range of the fractional part operation does not affect the truth of Theorem \ref{thm:berg.leib}; in particular, the calculations in \cite[\S5.9]{berg.leib} proceed in exactly the same way. Point 1 is therefore of no concern.

Point 2 is also of no concern, since the Mal'cev basis defined implicitly in \cite[Proposition 6.9]{berg.leib} is, in fact, nested. This is a consequence of the fact that the basis elements are taken in a \emph{legal order} in the sense of \cite[\S5.7]{berg.leib}.\footnote{Mal'cev bases in \cite{berg.leib} are, by definition, adapted to the lower central series \cite[\S1.2]{berg.leib}. However, taking the basis elements in a legel order in the sense of \cite[\S5.7]{berg.leib} does not guarantee that the resulting basis is adapted to the lower central series, as can be seen by considering the order defined in \cite[\S5.5]{berg.leib} in the case $d=4$. Being in a legal order does, however, guarantee that the basis is a nested Mal'cev basis in the sense we have defined in this paper.}

Point 3 is straightforward to overcome. So far, for each $i=1,\ldots,r$ we have a nilmanifold $T_{p_i}/Z_{p_i}$ of dimension $m_i$, say; a nested Mal'cev basis $\mathcal{Y}_i$ for $T_{p_i}/Z_{p_i}$ consisting of elements of the standard basis and their inverses; and a polynomial form $P_i$ on $T_{p_i}$ such that $\{\pm\Theta_i\}=\chi_{\mathcal{Y}_i}(P_i)_{m_i}$. We can define a nested Mal'cev basis $\mathcal{Y}'$ for the direct product $T_{p_1}\otimes\cdots\otimes T_{p_r}$ by simply taking the elements of $\mathcal{Y}_1$ in order, followed by the elements of $\mathcal{Y}_2$ in order, and so on up until we finally take the elements of $\mathcal{Y}_r$ in order. Note that $\chi_{\mathcal{Y}'}=(\chi_{\mathcal{Y}_1},\ldots,\chi_{\mathcal{Y}_r})$, and so in particular we have $\{-\Theta_i\}=\chi_{\mathcal{Y}'}(P_i)_{m_1+\ldots+m_i}$.

This leaves two further issues.
\begin{enumerate}\setcounter{enumi}{3}
\item Theorem \ref{thm:berg.leib} requires the $p_i$ all to be equal.
\item Theorem \ref{thm:berg.leib} requires that the $\Theta_i$ are expressed in terms of the \emph{last} $r$ coordinates of some polynomial form.
\end{enumerate}
We resolve point 4 really only for convenience in the main body of the paper. The proof of Theorem \ref{thm:bracket.non.uniform} would proceed almost identically in the event that $T_{p_1}\otimes\cdots\otimes T_{p_r}$ appeared in place of $T_p^r$ in the conclusion of Theorem \ref{thm:berg.leib}, but having $T_p^r$ makes some of our notation slightly cleaner. In fact, if we were concerned with optimising the implied constant in the conclusion of Theorem \ref{thm:bracket.non.uniform} then it would be preferable to allow $T_{p_1}\otimes\cdots\otimes T_{p_r}$ in place of $T_p^r$. However, we are not concerned with the exact bounds in Theorem \ref{thm:bracket.non.uniform}, and so we prove Theorem \ref{thm:berg.leib} as stated.

In any case, it is not difficult to obtain $T_p^r$ in place of $T_{p_1}\otimes\cdots\otimes T_{p_r}$. The key observation is that if $q$ is a multiple of $p$ then there is an obvious embedding $\iota_{p,q}:T_p\hookrightarrow T_q$ such that the image $\iota_{p,q}(Z_p)$ is a subset of $Z_q$. For example, the Heisenberg group $T_2$ embeds into $T_4$ via the map $\iota_{2,4}:T_2\hookrightarrow T_4$ defined by
\[
\iota_{2,4}\left(\begin{array}{ccc}
    1 & x & z\\
    0 & 1 & y\\
    0 & 0 & 1
    \end{array}\right)
=\left(\begin{array}{ccccc}
    1 & 0 & x & 0 & z\\
    0 & 1 & 0 & 0 & 0\\
    0 & 0 & 1 & 0 & y\\
    0 & 0 & 0 & 1 & 0\\
    0 & 0 & 0 & 0 & 1
    \end{array}\right),
\]
and the subgroup $\iota_{2,4}(Z_2)$ is equal to
\[
\left(\begin{array}{ccccc}
    1 & 0 & \Z & 0 & \Z\\
    0 & 1 & 0  & 0 & 0\\
    0 & 0 & 1  & 0 & \Z\\
    0 & 0 & 0  & 1 & 0\\
    0 & 0 & 0  & 0 & 1
    \end{array}\right).
\]
Moreover, and crucially, if $\mathcal{W}=\{W_1,\ldots,W_k\}$ is a nested Mal'cev basis for $T_p/Z_p$ consisting entirely of elements of the standard basis and their inverses, then it is possible to choose a nested Mal'cev basis $\mathcal{W}^{(q)}$ for $T_q/Z_q$ consisting entirely of elements of the standard basis and their inverses, and that includes the elements $W_i^{(q)}:=\log\iota_{p,q}(\exp W_i)$ in the same order that they appear in $\mathcal{W}$.\footnote{The basis $\mathcal{W}^{(q)}$ is not uniquely defined in this way. We simply choose a basis arbitrarily from all those nested Mal'cev bases consisting of elements of the standard basis and their inverses in which the elements $\log\iota_{p,q}(\exp W_i)$ appear in the desired order.} The upshot of this is that the non-zero coordinates of an element $\iota_{p,q}(g)$ with respect to $\mathcal{W}^{(q)}$ in $T_q$ will be the same as the coordinates of $g$ with respect to $\mathcal{W}$ in $T_p$. This implies that if $z\in Z_p$ and $gz$ belongs to the fundamental domain of $T_p/Z_p$ then $\iota_{p,q}(gz)$ belongs to the fundamental domain of $T_q/Z_q$, and so the non-zero entries of $\chi_{\mathcal{W}^{(q)}}(\iota_{p,q}(g))$ are equal to the non-zero entries of $\chi_\mathcal{W}(g)$.

Set $q$ as the lowest common multiple of the $p_i$, and write $m$ for the dimension of the group $T_q$. Set $P_i'=\iota_{p_i,q}(P_i)$ and $P=(P_1',\ldots,P_r')$. Define a basis $\mathcal{Y}$ for $T_q^r$ by taking the bases $\mathcal{Y}_i^{(q)}$ in order, but with the elements $Y_m^{(q)},\ldots,Y_m^{(q)}$ moved to the right so that they are now the last $r$ elements of the basis. Note that these basis elements are central, and so this last operation affects neither the property of being a nested Mal'cev basis nor the corresponding coordinates, and resolves point 5 above. We then have
\[
\{\pm\Theta_i\}=(\chi_{\mathcal{Y}}(P)_{rm-r+i}),
\]
as required by Theorem \ref{thm:berg.leib}.
\section{Basic properties of polynomial mappings}\label{ap:poly.map}
The main purpose of this appendix is to prove Lemma \ref{lem:plane.home}, which we now recall.
\newtheorem*{lem:plane.home}{Lemma \ref{lem:plane.home}}
\begin{lem:plane.home}
Let $k,p,r\in\N$. Then there is a some $d\in\N$ depending only on $k$ and $p$ such that if $\rho$ is an arbitrary polynomial mapping of degree at most $k$ into $T_p^r$ then the derivatives $\partial_{h_{d+1}}\ldots\partial_{h_1}\rho$ are all trivial.
\end{lem:plane.home}
The proof of Lemma \ref{lem:plane.home} rests on the following basic properties of polynomial mappings into $T_p^r$.
\begin{lemma}\label{lem:poly.map.properties}
Let $\rho,\sigma$ be polynomial mappings into $T_p$ of degree at most $k,k'$, respectively. Then
\begin{enumerate}
\renewcommand{\labelenumi}{(\roman{enumi})}
\item the product mapping $\rho\sigma$ taking $n$ to $\rho(n)\sigma(n)$ is a polynomial mapping of degree at most $k+k'$;
\item the inverse mapping $\rho^{-1}$ taking $n$ to $\rho(n)^{-1}$ is a polynomial mapping of degree $O_{k,p}(1)$.
\end{enumerate}
\end{lemma}
\begin{proof}
The first assertion is trivial. The second is also straightforward; we present the details for completeness.

We claim that each entry $(\rho^{-1}(n))_{ij}$ with $i<j\le p+1$ is a polynomial of degree at most $O_{k,i}(1)$, which is clearly sufficient to prove the second assertion. We prove this claim by induction on $p-i$; thus for any fixed $i$ we may assume that the claim holds for all values of $j$, for all greater values of $i$.

By definition of $\rho^{-1}$, for $i<j\le p+1$ we have
\[
\sum_{t=1}^p\rho(n)_{it}\rho^{-1}(n)_{tj}=0,
\]
but since the diagonal entries of each matrix $\rho(n)$ and $\rho(n)^{-1}$ are $1$, and the below-diagonal entries are $0$, this reduces to
\[
\rho(n)_{ij}+\rho^{-1}(n)_{ij}+\sum_{t=i+1}^{j-1}\rho(n)_{it}\rho^{-1}(n)_{tj}=0.
\]
This implies that
\[
\rho^{-1}(n)_{ij}=-\rho(n)_{ij}-\sum_{t=i+1}^{j-1}\rho(n)_{it}\rho^{-1}(n)_{tj},
\]
which is, by induction, a polynomial of degree at most $O_{k,i}(1)$, as claimed.
\end{proof}
\begin{proof}[Proof of Lemma \ref{lem:plane.home}]
It clearly suffices to prove the lemma in the case $r=1$.

Denote by $T_p(l)$ the subgroup of $T_p$ consisting of those matrices whose non-diagonal entries at a distance at most $l$ from the main diagonal are zero. Thus, for example, $T_p(0)=T_p$ and $T_p(p)=\{\id\}$. We claim that there is some $d\in\N$ depending only on $k$, $l$ and $p$ such that if $\rho$ is an arbitrary polynomial mapping of degree at most $k$ into $T_p$ whose image lies in $T_p(l)$ then the derivatives $\partial_{h_{d+1}}\ldots\partial_{h_1}\rho$ are all trivial. This is clearly sufficient to prove the lemma.

We prove this claim by induction on $p-l$; thus, for any fixed $l$, we may assume that the claim holds for all greater values of $l$.

The group operation of $T_p(l)$ restricted to the entries at a distance exactly $l+1$ from the main diagonal is simply addition in each entry. Therefore, if $\rho$ is a polynomial mapping of degree at most $k$ into $T_p$ whose image lies in $T_p(l)$, then every derivative $\partial_{h_{k+1}}\ldots\partial_{h_1}\rho$ lies in $T_p(l+1)$. Moreover, by Lemma \ref{lem:poly.map.properties} its other entries are all polynomials of degree at most $O_{k,p}(1)$, and so $\partial_{h_{k+1}}\ldots\partial_{h_1}\rho$ is a polynomial mapping of degree at most $O_{k,p}(1)$ whose image lies in $T_p(l+1)$. The claim, and hence the lemma, therefore follows by induction.
\end{proof}

\end{document}